\documentclass[11pt]{amsart}

\title{Hyperbolicity Beyond Convexity in the Alexandrov--Fenchel Inequality}

\author{Leo Brauner}
	\address{Institut f\"ur Mathematik \newline%
		\indent Goethe-Universit\"at Frankfurt \newline%
		\indent Robert-Mayer-Str. 10, 60325 Frankfurt am Main, Germany}
	\email{brauner@math.uni-frankfurt.de}
	
	\author{Oscar Ortega-Moreno}
	\address{Departamento de Matem\'aticas
        \newline%
		\indent CUNEF Universidad
        \newline%
		\indent Madrid, Spain}
	\email{oscar.ortegamoreno@cunef.edu}
    
\usepackage[english]{babel}
\usepackage[T1]{fontenc}
\usepackage{lmodern}
\usepackage{amsmath,amsthm,amssymb,amsfonts,mathrsfs}
\usepackage{thmtools}
\usepackage[abbrev,nobysame]{amsrefs}
\usepackage{float}
\usepackage{hyperref}
\usepackage[capitalise]{cleveref}
\usepackage[left=3.5cm,right=3.5cm,top=3cm,bottom=3cm,includeheadfoot]{geometry}
\usepackage{xcolor}
\usepackage{xcolor}
\usepackage{enumerate,enumitem}
\usepackage{bbm}
\usepackage{enumitem}
\usepackage{bm}

\hypersetup{
	colorlinks=true,
	linkcolor=teal,
	citecolor=orange
}

\BibSpec{article}{%
	+{}{\PrintAuthors}  		{author}
	+{,}{ \textit}     		{title}
	+{,}{ }             		{journal}
	+{}{ \textbf}       		{volume}
	+{}{ \parenthesize} 		{date}
	+{}{, no. } 		{number}
	+{,}{ }      	      		{conference}
	+{,}{ }      	      		{book}
	+{,}{ }            		{pages}
	+{,}{ }            	 	{note}
	+{,}{ }            	 	{status}
	+{,}{  \texttt } {eprint}
	+{.}{}              {transition}
}

\BibSpec{book}{%
	+{}{\PrintAuthors}  {author}
	+{,}{ \textit}      {title}
	+{,}{ }             {publisher}
	+{,}{ }             {place}
	+{,}{ }             {date}
	+{.}{}              {transition}
}

\BibSpec{incollection}{%
	+{}  {\PrintAuthors}                {author}
	+{,} { \textit}                     {title}
	+{.} { }                            {part}
	+{:} { \textit}                     {subtitle}
	+{,} { \PrintContributions}         {contribution}
	+{,} { \PrintConference}            {conference}
	+{,} { }                            {booktitle}
	+{,} { pp.~}                        {pages}
	+{,}{ }       		                {series}
	+{}{ \textbf}                         {volume}
	+{,} { }                            {publisher}
	+{,} { }                 {date}
	+{,} { }                            {status}
	+{,} { \PrintDOI}                   {doi}
	+{,} { available at \eprint}        {eprint}
	+{}  { \parenthesize}               {language}
	+{}  { \PrintTranslation}           {translation}
	+{;} { \PrintReprint}               {reprint}
	+{.} { }                            {note}
	+{.} {}                             {transition}
}

\theoremstyle{plain}
\newtheorem{thm}{Theorem}[section]
\newtheorem*{thm*}{Theorem}
\newtheorem{lem}[thm]{Lemma}
\newtheorem{prop}[thm]{Proposition}

\crefname{thm}{Theorem}{Theorems}
\crefname{lem}{Lemma}{Lemmas}
\crefname{prop}{Proposition}{Propositions}
\crefname{cor}{Corollary}{Corollaries}

\newtheorem{thmintro}{Theorem}
\newtheorem*{thmintro*}{Theorem}

\newtheorem*{corintro*}{Corollary}

\crefname{thmintro}{Theorem}{Theorems}
\crefname{thmintro*}{Theorem}{Theorems}
\crefname{corintro}{Corollary}{Corollaries}
\crefname{corintro*}{Corollary}{Corollaries}

\theoremstyle{definition}
\newtheorem{defi}[thm]{Definition}

\crefname{defi}{Definition}{Definitions}

\theoremstyle{remark}
\newtheorem{exl}[thm]{Example}
\newtheorem{rem}[thm]{Remark}

\crefname{exl}{Example}{Examples}
\crefname{rem}{Remark}{Remarks}

\numberwithin{equation}{section}
\crefname{equation}{}{}

\newcommand{\N}{\mathbb{N}}
\newcommand{\R}{\mathbb{R}}
\renewcommand{\S}{\mathbb{S}}

\newcommand{\K}{\mathcal{K}}

\newcommand{\MSO}{\mathrm{M}}

\newcommand{\BB}{\mathbb{B}}

\newcommand{\pair}[2]{\langle #1,#2 \rangle}

\newcommand{\Sym}{\mathrm{Sym}}

\DeclareMathOperator{\tr}{tr}

\DeclareMathOperator{\diag}{diag}
\DeclareMathOperator{\rad}{rad}

\newcommand{\D}{\mathsf{D}}
\newcommand{\V}{\mathsf{V}}

\begin{document}

    \begin{abstract}
    	We establish an Alexandrov--Fenchel type inequality in which one of the
    	reference bodies is replaced by a function whose Hessian satisfies a balancing condition on its eigenvalues, without being required to be positive semidefinite.
    	The proof hinges on a new extension of Alexandrov's inequality for mixed
    	discriminants. As applications, we derive a log-concavity principle for
    	linear functionals of area measures, answering a question of Colesanti,
    	Hug, and Saorín-Gómez, as well as Brunn--Minkowski type inequalities for
    	intrinsic volumes of mean section bodies, answering a question of Schuster.
	\end{abstract}
	
    \maketitle

    \thispagestyle{empty}

    \section{Introduction}
    
    Let $\K^n$ denote the space of convex bodies, that is, nonempty compact convex subsets of Euclidean space $\R^n$. A classical theorem of Minkowski states that the volume of a Minkowski combination is a homogeneous polynomial of degree $n$: for all $m\in\N$, $K_1,\ldots,K_m\in\K^n$, and $\lambda_1,\ldots,\lambda_m\geq 0$,
    \begin{equation*}
    	\mathrm{vol}_n(\lambda_1K_1+\cdots+\lambda_mK_m)
    	=
    	\sum_{i_1,\ldots,i_n\in[m]}
    	\V(K_{i_1},\ldots,K_{i_n})
    	\lambda_{i_1}\cdots\lambda_{i_n}.
    \end{equation*}
    The coefficients $\V(K_{i_1},\ldots,K_{i_n})$ are the \emph{mixed volumes}. Many of the fundamental geometric quantities associated with a convex body arise as special cases of mixed volumes. In particular, volume, surface area, and mean width are obtained in this way, and more generally the intrinsic volumes are, up to dimensional constants, given by $\V(K^{[j]},{\BB}^{[n-j]})$, where $K$ appears $j$ times and the Euclidean unit ball $\BB$ appears $n-j$ times.  We refer to Schneider's monograph~\cite{Schneider2014} for a comprehensive account of the theory of mixed volumes.
    
    Inequalities between mixed volumes encode a wide range of geometric information about convex bodies. Arguably the most fundamental among these inequalities is the Alexandrov--Fenchel inequality. First proved by Alexandrov~\cites{Aleksandrov1937,Alexandroff1938}, it states that for all $K,L,M_1,\ldots,M_{n-2}\in\K^n$,
    \begin{equation}\label{eq:Alexandrov-Fenchel_ineq}
    	\V(K,L,M_1,\ldots,M_{n-2})^2
    	\geq
    	\V(K,K,M_1,\ldots,M_{n-2})\,
    	\V(L,L,M_1,\ldots,M_{n-2}).
    \end{equation}
    The importance of this inequality lies in the range of classical results that follow from it such as the Brunn--Minkowski inequality, Minkowski's inequalities for mixed volumes, and the isoperimetric inequality for convex bodies. 
    
    Despite its classical origins, the inequality remains a focus of current research. For instance, Shenfeld and van Handel~\cite{Shenfeld2019} gave a concise modern proof, and their subsequent study of its equality cases led to two major breakthroughs~\cites{Shenfeld2020,Shenfeld2022}, although a complete characterization remains open in full generality (see also \cites{Hug2024c} for a generalization of these results to a broader class of bodies). Extensions of the inequality continue to be actively studied as well; for instance, see \cites{Kotrbaty2021,Kotrbaty2022,Bernig2024} for the development of Hodge--Riemann relations in the setting of valuations on convex bodies.
    
    A fundamental point of view, going back to the work of Minkowski, is to regard the Alexandrov--Fenchel inequality as a statement about quadratic forms: for fixed reference bodies $M_1,\ldots,M_{n-2}$, the Alexandrov--Fenchel inequality may be viewed as a statement about the quadratic form
    \[
    (h_K,h_L)\mapsto \V(K,L,M_1,\ldots,M_{n-2}),
    \] 
    where $h_K$  denotes the support function of $K\in \K^n$. Hilbert~\cite{Hilbert1912} developed this perspective further, providing an analytic method to it in dimension $n=3$.
    More precisely, Hilbert observed that, the inequality can be reduced to a statement about the signature of the associated quadratic form. His key idea was to show that this quadratic form is hyperbolic: it is positive on some line and negative semidefinite on a codimension-one subspace of $C^2(\S^{n-1})$. Thus, the Alexandrov--Fenchel inequality can be viewed as a consequence of this spectral signature. The same idea appears in Alexandrov's proof of the general inequality via his mixed discriminant inequality~\cite{Alexandroff1938}, as well as in the more recent proof of Shenfeld and van Handel~\cite{Shenfeld2019}.
    
    The question we address in this paper concerns extending the Alexandrov--Fenchel
    inequality beyond mixed volumes of convex bodies, a topic on which relatively little appears to be known apart from a few very special results and a general structural theorem; see, e.g.~\cites{Abardia2015,Ross2025}. More precisely, we
    ask whether some of the reference bodies can be replaced by general  functions on $\S^{n-1}$ while preserving the hyperbolicity of the resulting quadratic form. To make this precise, recall that every $f\in C^2(\S^{n-1})$ can be written
    as a difference of two support functions. Minkowski multilinearity therefore
    extends the mixed volume uniquely to a multilinear form on
    $C^2(\S^{n-1})$. Consequently, we may consider the quadratic form
    $$
    (f,g)\mapsto \V(f,g,\zeta,M_2,\ldots,M_{n-2}),
    $$
    where a convex body is replaced by an arbitrary $C^2$ function, and ask for which functions $\zeta$ it is hyperbolic. The main purpose of this article is to show that this is possible under a spectral condition on the Hessian of $\zeta$, provided that some of the remaining reference bodies are Euclidean balls.
    
    To state our main result, we introduce one piece of notation. For
    $\zeta\in C^2(\S^{n-1})$ and $u\in\S^{n-1}$, let $D^2\zeta(u)$ denote the Hessian of the one-homogeneous extension of $\zeta$, restricted to
    $u^\perp$. Recall that $\zeta$ is a support function if and only if
    $D^2\zeta(u)$ is positive semidefinite at every point $u\in\S^{n-1}$. 
    
    We now introduce the spectral condition that will allow us to extend this hyperbolicity phenomenon beyond the setting of convexity. If $Z$ is a symmetric matrix with ordered eigenvalues $ \lambda_1\leq\cdots\leq\lambda_n$ (counted with multiplicity),    
    we say that $Z$ is \emph{spectrally $k$-balanced} if
    $$
    \lambda_1+\cdots+\lambda_{k+1}\geq\lambda_n.
    $$
    Roughly speaking, this condition requires the smaller eigenvalues, taken collectively, to be large enough to control the largest one. We can now state our main result.
    
    \begin{thmintro}\label{thmintro:Alexandrov-Fenchel}
    	Let $2<j\leq n-1$ and $\zeta\in C^2(\S^{n-1})$ be such that, for every
    	$u\in\S^{n-1}$, $D^2\zeta(u)$ is positive semidefinite or
    	spectrally $(n-j)$-balanced. Then, for all
    	$K,L\in\K^n$ and $\bm M=(M_1,\ldots,M_{j-2})\subseteq\K^n$,
    	\begin{align}\label{eq:thmintro:Alexandrov-Fenchel}
    		\V(K,L,\bm M,\zeta,{\BB}^{[n-j-1]})^2
    		&\geq
    		\V(K,K,\bm M,\zeta,{\BB}^{[n-j-1]})
    		\V(L,L,\bm M,\zeta,{\BB}^{[n-j-1]}).
    	\end{align}
    \end{thmintro}
        
    If $\zeta=h_{\BB}$, then $D^2\zeta(u)=I_{u^\perp}$, and~\cref{eq:thmintro:Alexandrov-Fenchel} is the classical Alexandrov--Fenchel inequality for $K,L,\bm{M}$ and $n-j$ copies of the unit ball. More generally, positivity of $D^2\zeta$ reduces the result to the classical case. The novelty is that $D^2\zeta(u)$ may be indefinite, provided it is spectrally $(n-j)$-balanced. Since spectral $k$-balance implies spectral $(k+1)$-balance, the condition weakens as $n-j$ increases; for $j=n-1$ it forces positive semidefiniteness, so the genuinely new range is $3\leq j\leq n-2$. The exclusion of $j=2$ is an artifact of our method; this case is covered by a different argument in forthcoming work of Brauner, Ortega-Moreno, and Schuster~\cite{Brauner2026}.
    
    The significance of \cref{thmintro:Alexandrov-Fenchel} is not merely that it extends the Alexandrov--Fenchel inequality beyond support functions. Its greater flexibility has concrete geometric consequences. As in the classical theory, it yields new Minkowski and Brunn--Minkowski type inequalities for a large family of geometric operators. In particular, our applications include new Brunn--Minkowski type inequalities for mean section operators conjectured by Schuster, showing that \cref{thmintro:Alexandrov-Fenchel} captures genuinely new geometric phenomena.
    
    \subsection*{A mixed discriminant inequality}

    Our proof follows the same general strategy as Hilbert's: we recast the desired reverse Cauchy--Schwarz inequality as a statement about the spectral signature of an associated quadratic form. To implement this approach in the present setting, we develop a version of Alexandrov's mixed discriminant inequality adapted to matrices satisfying the spectral balance condition above. This extension is the main local ingredient that allows Hilbert's method to go beyond the classical setting of support functions.
    
    The mixed discriminant polarizes the determinant in the same way that the mixed volume polarizes the volume of convex bodies.
    Here and in the following, we denote by $\Sym^2(\R^n)$ the vector space of symmetric $n$-dimensional matrices, by $\Sym^2_+(\R^n)$ the subset of positive semidefinite matrices, and by $I_n$ the $n$-dimensional identity matrix.
    
    As was shown by Alexandrov~\cite{Alexandroff1938}, mixed discriminants satisfy a reverse Cauchy--Schwarz inequality that mirrors~\cref{eq:Alexandrov-Fenchel_ineq}: for all matrices $A, B , C_1,\ldots,C_{n-2} \in \Sym^2_{+}(\R^n)$,
    \begin{equation}\label{eq:Alexandrov_matrix_ineq}
    	\D(A,B,C_1,\ldots,C_{n-2})^2
    	\geq \D(A,A,C_1,\ldots,C_{n-2})\D(B,B,C_1,\ldots,C_{n-2}).
    \end{equation}
    This inequality lies at the heart of several proofs of the Alexandrov--Fenchel inequality. 
    
    Our second main result generalizes~\cref{eq:Alexandrov_matrix_ineq} in the same way that \cref{thmintro:Alexandrov-Fenchel} generalizes~\cref{eq:Alexandrov-Fenchel_ineq}: one of the matrices is replaced by a spectrally balanced matrix that need not be positive semidefinite, at the cost of $n-j-1$ copies of the identity matrix.
    
    \begin{thmintro}\label{thmintro:mixed_discr_ineq}
    	Let $2\leq j\leq n-1$ and $Z\in\Sym^2(\R^n)$ be spectrally $(n-j)$-balanced. Then for all matrices $A,B\in\Sym^2_{+}(\R^n)$ and $\bm{C}=(C_1,\ldots,C_{j-2}) \subseteq \Sym^2_{+}(\R^{n})$,
    	\begin{align}\label{eq:thmintro:mixed_discr_ineq}
    		\D(A,B,\bm{C},Z,I_n^{[n-j-1]})^2
    		\geq \D(A,A,\bm{C},Z,I_n^{[n-j-1]}) \D(B,B,\bm{C},Z,I_n^{[n-j-1]}).
    	\end{align}
    \end{thmintro}
    
    The two theorems are related as follows. At each point $u\in\S^{n-1}$, the proof of \cref{thmintro:Alexandrov-Fenchel} requires \cref{thmintro:mixed_discr_ineq} in the $(n-1)$-dimensional space $u^\perp$, with $j-1$ in place of $j$ and $Z=D^2\zeta(u)$. Since $n-j=(n-1)-(j-1)$, the balance condition in \cref{thmintro:mixed_discr_ineq} is exactly the one imposed on $D^2\zeta(u)$ in \cref{thmintro:Alexandrov-Fenchel}. As \cref{thmintro:mixed_discr_ineq} is only available for $j-1\geq 2$, this is where the restriction $j>2$ in \cref{thmintro:Alexandrov-Fenchel} originates; for $j=2$, the pointwise inequality that our approach would require fails (see \cref{rem:homogeneity_2}).
    
    Note that the set of spectrally $(n-j)$-balanced matrices is a closed convex cone in $\Sym^2(\R^n)$. It has nontrivial intersection with the cone $\Sym^2_+(\R^n)$ but, unless $j=n-1$, is not contained in it. The proof of \cref{eq:thmintro:mixed_discr_ineq} is by an induction on $j\geq 2$, where the induction step involves a reduction of the dimension, similar to Alexandrov's original argument.
    
    
    \subsection*{Application I: Functionals of area measures}

    The \emph{area measure} of order $j\in\{1,\ldots,n-1\}$ of a convex body $K\in\K^n$, denoted by $S_j(K,{}\cdot{})$, is a classical geometric measure on the unit sphere $\S^{n-1}$. If $K$ has smooth boundary, $S_j(K,{}\cdot{})$ is given in terms of the Gauss map by the $j$-th elementary symmetric function of the reciprocal principal curvatures. If $K$ is a polytope, $S_j(K,{}\cdot{})$ is supported on the normal directions to the $j$-dimensional faces of $K$. For a given function $\zeta\in C^2(\S^{n-1})$, define $\mathcal{F}_{\zeta} : \K^n \to \R$ by
    \begin{equation}\label{eq:functional_area_measure}
    	 \mathcal{F}_{\zeta}(K):=\frac{1}{n}\int_{\S^{n-1}} \zeta(u)\, dS_j(K,u).
    \end{equation}
    The functional $\mathcal{F}_\zeta$ can be written as a mixed volume, interpeted as a multilinear form on $C^2(\S^{n-1})$:
    \begin{equation}\label{eq:F_zeta_mixed_vol}
    	\mathcal{F}_\zeta(K) = \V(K^{[j]},\zeta,{\BB}^{[n-j-1]}).
    	\qquad K\in\K^n.
    \end{equation}
    Whenever $\zeta$ is a support function, this is of course a mixed volume of convex bodies.
    
    Colesanti, Hug and Saorín-Gómez~\cites{Colesanti2014,Colesanti2017e} showed that the functional $\mathcal{F}_\zeta$ is monotone (with respect to set inclusion) if and only if at every point $u\in\S^{n-1}$, the matrix $D^2\zeta(u)$ is \emph{spectrally $(n-j)$-positive}, meaning that the sum of any $n-j$ of its eigenvalues is nonnegative. Moreover, they showed that spectral positivity is necessary for log-concavity.
    
    \begin{thm*}[{\cites{Colesanti2014,Colesanti2017e}}]
    	Let $2 \leq j \leq n-1$ and $\zeta\in C^2(\S^{n-1})$. If $\mathcal{F}_{\zeta}$ is log-concave, then $D^2\zeta(u)$ is spectrally $(n-j)$-positive for every $u\in\S^{n-1}$.
    \end{thm*}

    In~\cite{Colesanti2017e}, the authors raise the question of finding natural sufficient conditions on $\zeta$, other than being a support function, under which $\mathcal{F}_{\zeta}$ is log-concave. In the light of~\cref{eq:F_zeta_mixed_vol}, \cref{thmintro:Alexandrov-Fenchel} provides such a condition.
    
    \begin{thmintro}\label{thmintro:Brunn-Minkowski}
    	Let $2 < j \leq n-1$ and $\zeta\in C^2(\S^{n-1})$ be such that the matrix $D^2\zeta(u)$ is positive semidefinite or spectrally $(n-j)$-balanced at every point $u\in\S^{n-1}$. Then for all $K,L\in\K^n$ and $\lambda \in (0,1)$,
    	\begin{equation}\label{eq:thmintro:Brunn-Minkowski}
    		\mathcal{F}_{\zeta}((1-\lambda)K+\lambda L)
    		\geq \mathcal{F}_{\zeta}(K)^{1-\lambda} \mathcal{F}_{\zeta}(L)^{\lambda}.
    	\end{equation}
    \end{thmintro}
    
    Since spectral balance implies spectral positivity, \cref{thmintro:Brunn-Minkowski} gives a sufficient condition that strengthens the necessary condition above. The authors believe that the condition in \cref{eq:thmintro:Brunn-Minkowski} can be further relaxed. Whether spectral positivity alone suffices remains open.
    
    \subsection*{Application II: Mean section operators}
    
    As a second application, we obtain a family of Brunn--Minkowski type inequalities for a Minkowski mean of the $k$-dimensional sections of the convex bodies involved. Given $k\in\{0,\ldots,n\}$ and a convex body $K\in\K^n$, Goodey and Weil~\cite{Goodey1992} defined its \emph{$k$-th mean section body} $\MSO_k K$ by
    \begin{equation}\label{eq:MSO_defi}
    	h(\MSO_k K,u)
    	= \int_{\mathrm{AGr}_k(\R^n)} h(K\cap E,u)\, d\mu_k(E),
    	\qquad u\in\S^{n-1},
    \end{equation}
    where $\mathrm{AGr}_k(\R^n)$ denotes the affine Grassmannian and $\mu_k$ is a suitably normalized motion invariant measure on $\mathrm{AGr}_k(\R^n)$.
    For $1\leq k\leq n$, the \emph{$k$-th mean section operator} $\MSO_k :\K^n \to \K^n$ is continuous with respect to the Hausdorff metric, homogeneous of degree~$n-k+1$, and commutes with all rotations. In the last decades, these operators have attracted considerable attention with regard to questions of injectivity, fixed points, isoperimetric and Minkowski type problems, as well as connections to integral geometry (see \cites{Berg2018,Goodey1992,Goodey2012,Goodey2014,Brauner2023,Brauner2024a,Brauner2024,Haberl2019,Hofstaetter2025}).

    Here we prove the following theorem, which gives an almost complete answer to a question posed by Schuster~\cite{Schuster2006} two decades ago:
    
    \begin{thmintro}\label{thmintro:B-M_mean_sec}
    	Let $1\leq i\leq n$ and $2\leq k < n-1$. Then for all $K,L\in\K^n$ and $\lambda\in (0,1)$,
    	\begin{equation}\label{eq:thmintro:B-M_mean_sec}
    		V_i(\MSO_k((1-\lambda)K+\lambda L))
    		\geq V_i(\MSO_k K)^{1-\lambda} V_i(\MSO_k L)^{\lambda}.
    	\end{equation}
    	If $K$ and $L$ are full-dimensional, equality holds if and only if they are translates.
    \end{thmintro}
    
    Brunn--Minkowski type inequalities for geometric operators were first established by Lutwak~\cite{Lutwak1993a} for projection body maps, and his technique has since been developed further to handle other operators, including mean section operators. Inequality~\cref{eq:thmintro:B-M_mean_sec} was first shown by Schuster~\cite{Schuster2006} for $k=2$ and all $1\leq i\leq n$. It was later shown by Schuster~\cite{Schuster2010} to hold for $i=n-k+2$ and then extended by Parapatits and Schuster~\cite{Parapatits2012} to the range $i\leq n-k+2$. Hence, the new cases covered by \cref{thmintro:B-M_mean_sec} are $i>n-k+2$ and $2\leq k<n-1$. The case $k=n-1$ of means of hyperplane sections corresponds to $j=2$ above and is not covered by our approach.

    One important ingredient in our proof is an integral representation of mean section operators that was not available when these earlier results were established.
    Goodey and Weil~\cites{Goodey1992,Goodey2014} showed that for $2\leq k\leq n-1$ and every convex body $K\in\K^n$ of dimension at least $n-k+2$, the support function of $\MSO_k K$, up to a translation and dilation, is given by
    \begin{equation}
    	\int_{\S^{n-1}} g_k(\pair{u}{v}) \, dS_{n-k+1}(K,v),
    	\qquad u\in\S^{n-1},
    \end{equation}
    where the function $g_k \in C^{\infty}(-1,1)$ arises from the Green's function of some second-order differential operator on the unit sphere in $\R^k$. 
     Combining~\cref{eq:MSO_integral_rep} with a technique of Lutwak~\cite{Lutwak1990}, we express mixed volumes involving $\MSO_k K$ as $\V(K^{[j]},\zeta,{\BB}^{[n-j-1]})$, where $j=n-k+1$. This reduces~\cref{eq:thmintro:B-M_mean_sec} to \cref{thmintro:Brunn-Minkowski} for functions of the form
    \begin{equation*}
    	\zeta(u)
    	= \int_{\S^{n-1}} g_k(\pair{u}{v})f(v)\, dv,\qquad u\in\S^{n-1},
    \end{equation*}
    where $f\in C(\S^{n-1})$ is nonnegative and orthogonal to all linear functions. Through a careful analysis of Berg's functions, we show that $\zeta\in C^2(\S^{n-1})$ and that $D^2\zeta(u)$ is spectrally $(k-1)$-balanced for every $u\in\S^{n-1}$.

    \subsection*{Organization of the article}
    
    In \cref{sec:hyperbolicity}, we discuss hyperbolic quadratic forms, which provide the general mechanism behind reverse Cauchy--Schwarz inequalities. In \cref{sec:spec_pos_and_balanced}, we collect basic properties of spectrally positive and balanced matrices. We employ all of this in \cref{sec:mixed_discriminants} to prove the mixed discriminant inequality of \cref{thmintro:mixed_discr_ineq}. \cref{sec:mixed_vols} is devoted to mixed volume inequalities: we establish the Alexandrov--Fenchel type inequality of \cref{thmintro:Alexandrov-Fenchel} and derive Minkowski and Brunn--Minkowski type inequalities from it, in particular \cref{thmintro:Brunn-Minkowski}. Finally, in \cref{sec:mean_sec_ops}, we apply these results to mean section operators and prove \cref{thmintro:B-M_mean_sec}.

    \section{Hyperbolic quadratic forms}
    \label{sec:hyperbolicity}

    In this section, we discuss hyperbolic quadratic forms in an abstract setting. They provide the natural framework for understanding the reverse Cauchy--Schwarz inequalities that arise in Alexandrov's mixed discriminant inequality and the Alexandrov--Fenchel inequality. We begin by recalling the definition of hyperbolicity.

    \begin{defi}\label{defi:hyperbolic}
        A quadratic form $Q$ on a real vector space $X$ is \emph{hyperbolic} if it is negative semidefinite on a hyperplane of $X$ but not on the entire space $X$.
    \end{defi}

    Throughout this article we will identify a quadratic form with its associated bilinear form (that is, $Q(x,x)=Q(x)$).
    For the remaining section, we always assume that $X$ is a real vector space and $Q$ is a quadratic form on $X$.
    We want to emphasize that the theory laid out in this section does not require $X$ to be finite-dimensional nor to carry an inner-product structure.
    The following lemma (which can be regarded as a modification of \cite{Shenfeld2019}*{Lemma~2.9}) demonstrates the connection between hyperbolicity and reverse Cauchy--Schwarz inequalities.

    \begin{lem}\label{lem:hyperbolic_iff}
        The following are equivalent.
        \begin{enumerate}[label=\upshape(\alph*)]
			\item \label{lem:hyperbolic_iff:reverse_C-S}
            $Q(x,y)^2\geq Q(x,x)Q(y,y)$ for all $x,y \in X$ such that $Q(y,y) > 0$.
			\item \label{lem:hyperbolic_iff:hyperbolic}
            There exists a linear form $\xi : X \to \mathbb{R}$ such that $Q(x,x)\leq 0$ for all $x \in \ker\xi$.
		\end{enumerate}
    \end{lem}
    \begin{proof}
        \ref*{lem:hyperbolic_iff:reverse_C-S}$\Rightarrow$\ref*{lem:hyperbolic_iff:hyperbolic}:
        If $Q$ is negative semidefinite, we may choose $\xi=0$. Otherwise, there exists a vector $y\in X$ such that $Q(y,y)>0$. The reverse Cauchy--Schwarz inequality asserts that $Q(x,x)\leq 0$ whenever $Q(x,y)=0$. Hence, the linear form $\xi :X\to\R:\xi(x)=Q(x,y)$ is such that $Q$ is negative semidefinie on $\ker\xi$.

        \ref*{lem:hyperbolic_iff:hyperbolic}$\Rightarrow$\ref*{lem:hyperbolic_iff:reverse_C-S}:
        Take $x,y\in X$ such that $Q(y,y)>0$. Then $\xi(y)\neq 0$ and $x - a y \in \ker\xi$, where we defined $a=\xi(x)/\xi(y)$. By our assumption,
        \begin{align*}
            \begin{pmatrix}
                1 & - a
            \end{pmatrix}
            \begin{pmatrix}
                Q(x,x) & Q(x,y) \\
                Q(y,x) & Q(y,y)
            \end{pmatrix}
            \begin{pmatrix}
                1 \\ - a
            \end{pmatrix}
            = Q(x - ay,x - ay)
            \leq 0.
        \end{align*}
        By Sylvester's criterion, the two-by-two matrix above is negative semidefinite, and thus,
        \begin{align*}
            Q(x,y)^2 - Q(x,x)Q(y,y)
            =  - \det
            \begin{pmatrix}
                Q(x,x) & Q(x,y) \\
                Q(y,x) & Q(y,y)
            \end{pmatrix}
            \geq 0.
        \end{align*}
        This is precisely the desired reverse Cauchy--Schwarz inequality.
    \end{proof}

    In particular, if $Q$ is not negative semidefinite on $X$, then $Q$ satisfies a reverse Cauchy--Schwarz inequality if and only if it is hyperbolic. In the finite-dimensional case, this can also be characterized in terms of the signature of $Q$.

    \begin{lem}\label{lem:hyperbolic_iff_signature}
        If $X$ is finite-dimensional, then $Q$ is hyperbolic if and only if it has signature $(+,-,\ldots,-)$.
    \end{lem}
    \begin{proof}
        Let throughout $\dim X=n+1$.
        Suppose first that $Q$ is hyperbolic and choose a vector $y\in X$ and a linear form $\xi : X \to \mathbb{R}$ such that $Q(y,y)=1$ and $Q$ is negative semidefinite on $\ker\xi$. Then there exists a basis $\{b_1,\ldots,b_n\}$ of the hyperplane $\ker \xi$ such that $Q(b_i,b_j)=-\delta_{ij}$ for $i,j\in [n]$. Define a vector
        \begin{equation*}
            b_0 = y + \sum_{i\in [n]} Q(y,b_i)b_i.
        \end{equation*}
        Direct computation shows that $Q(b_0,b_i)=0$ for all $i\in [n]$ and $Q(b_0,b_0)=1$. This implies that $b_0\notin \ker\xi$, and thus, $\{b_0,b_1,\ldots,b_n\}$ is a basis of $X$. Since $(Q(b_i,b_j))_{i,j} = \diag(+1,-1,\ldots,-1)$, the signature of $Q$ is given by $(+,-,\ldots,-)$.

        Conversely, suppose now that $Q$ has signature $(+,-,\ldots,-)$. Then there exists a basis $\{\xi_0,\xi_1,\ldots,\xi_n\}$ of the dual space $X^{\ast}$ such that $Q$ can be expressed as
        \begin{equation*}
            Q(x,x) = \xi_0(x)^2 - \sum_{i\in [n]} \xi_i(x)^2,
            \qquad x\in X.
        \end{equation*}
        Clearly $Q$ is negative semidefinite on $\ker\xi_0$, but not on the entire space $X$.
    \end{proof}

    In practice, one is often interested in reverse Cauchy--Schwarz inequalities on a mere convex cone $\mathcal{C}\subseteq X$. The following lemma (taken from \cite{Burago1988}*{Proposition~27.2.4}) asserts that this is essentially still equivalent to hyperbolicity. In there, we say that $y$ is in the interior of $\mathcal{C}$ if for every $x\in X$, there exists $\varepsilon>0$ such that $y+(-\varepsilon,\varepsilon)x \subseteq \mathcal{C}$.

    \begin{lem}[{\cite{Burago1988}}]\label{lem:hyperbolic_on_cone}
        Let $\mathcal{C}\subseteq X$ be a convex cone and suppose there exists some $y$ in the interior of $\mathcal{C}$ such that $Q(y,y)>0$ and $Q(x,y)^2\geq Q(x,x)Q(y,y)$  for all $x\in \mathcal{C}$. Then $Q$ is hyperbolic.
    \end{lem}
    \begin{proof}
        Define $\xi:X\to \R: \xi(x)=Q(x,y)$ and take $x\in \ker \xi$. Since $y$ lies in the interior of $\mathcal{C}$, we may choose some $t\neq 0$  such that $y+tx\in\mathcal{C}$. Next, define $z=x+\frac{1}{t}y \in\mathcal{C}$ and $a=Q(z,y)/Q(y,y)$ and note that $x=z-ay$. Then we obtain
        \begin{align*}
            Q(x,x)
            = Q(z,z) - 2aQ(z,y)+ a^2 Q(y,y)
            = Q(z,z) - \frac{Q(z,y)^2}{Q(y,y)}
            \leq 0.
        \end{align*}
        Hence, $Q$ is negative semidefinite on $\ker\xi$ but not on $X$.
    \end{proof}

    In light of this lemma, we may interpret the classical quadratic mixed discriminant and mixed volume inequalities of Alexandrov in terms of hyperbolic quadratic forms.

    \begin{exl}\label{exl:Alexandrov_mixed_disxr_hyperbolic}
        In the case of Alexandrov's mixed discriminant inequality~\cref{eq:Alexandrov_matrix_ineq}, we may consider the vector space $X=\Sym^2(\R^n)$ of symmetric $n$-dimensional matrices and the convex cone $\mathcal{C}=\Sym^2_+(\R^n)$ of positive semidefinite matrices. It is easy to see that the identity matrix $I_n$ is in the interior of $\Sym_+^2(\R^n)$. For fixed reference matrices $C_1,\ldots,C_{n-2}\in\Sym_+^2(\R^n)$, the quadratic form $Q=\D({}\cdot{},{}\cdot{},C_1,\ldots,C_{n-2})$ is hyperbolic on $\Sym^2(\R^n)$.
    \end{exl}

    \begin{exl}\label{exl:A-F_ineq_hyperbolic}
        The Alexandrov--Fenchel inequality~\cref{eq:Alexandrov-Fenchel_ineq} can be obtained by approximation from convex bodies that are sufficiently smooth. We may consider the vector space $X=C^2(\S^{n-1})$ and the convex cone $\mathcal{C}\subseteq X$ of $C^2$~support functions. It is a well-known fact (see~\cref{sec:mixed_vols:preps}) that the support function of the unit ball is in the interior of $\mathcal{C}$. In fact, any $C^2(\S^{n-1})$~function can be written as the difference of $C^2$~support functions, and thus, the mixed volume extends to a multilinear form on $C^2(\S^{n-1})$. For fixed reference bodies $M_1,\ldots,M_{n-2}\in \K^n$ with $C^2$~support functions, the quadratic form $Q=\V({}\cdot{},{}\cdot{},M_1,\ldots,M_{n-2})$ is hyperbolic on $C^2(\S^{n-1})$.
    \end{exl}

    For our purposes, it will be also important to control the equality conditions in the reverse Cauchy--Schwarz inequality.

    \begin{defi}
         The \emph{radical} of $Q$ is the linear subspace $\rad Q\subseteq X$ defined as
         \begin{equation*}
            \rad Q = \{x\in X \colon Q(x,y)=0 \text{ for all } y\in X\}.
        \end{equation*}
        Moreover, $Q$ is called \emph{nondegenerate} if $\rad Q =\{0\}$.
    \end{defi}

    The next lemma shows another feature of the perspective of hyperbolic quadratic forms: namely, that the equality conditions of the reverse Cauchy--Schwarz inequality can be described straightforwardly in terms of radical.

    \begin{lem}\label{lem:hyperbolic_equal_case}
        Let $Q$ be hyperbolic. Then for all $x,y\in X$ such that $Q(y,y)>0$, the following are equivalent.
        \begin{enumerate}[label=\upshape(\alph*)]
            \item \label{lem:hyperbolic_equal_case:equality}
            $Q(x,y)^2 = Q(x,x)Q(y,y)$.
            \item \label{lem:hyperbolic_equal_case:radical}
            There exists some $a\in\R$ such that $x-ay\in\rad Q$.
        \end{enumerate}
        In this case, $a\in\R$ is unique and given by $ a = Q(x,y) / Q(y,y) $.
    \end{lem}
    \begin{proof}
        \ref*{lem:hyperbolic_equal_case:equality} $\Rightarrow$ \ref*{lem:hyperbolic_equal_case:radical}: Suppose that $Q(x,y)^2 = Q(x,x)Q(y,y)$ and set $ a = Q(x,y) / Q(y,y)$. Then we have
        \begin{equation*}
            Q(x-ay,y) = 0
            \qquad\text{and}\qquad
            Q(x-ay,x-ay) = Q(x,x) - \frac{Q(x,y)^2}{Q(y,y)} = 0.
        \end{equation*}
        By hyperbolicity, the bilinear form $Q$ is negative semidefinite on the hyperplane $\{z\in X: Q(z,y)=0\}\subseteq X$. Hence, for every $z\in X$ such that $Q(z,y)=0$, the Cauchy--Schwarz inequality yields
        \begin{equation*}
            Q(x-ay,z) ^2 \leq Q(x-ay,x-ay)Q(z,z) = 0,
        \end{equation*}
        and thus, $Q(x-ay,z)=0$. Now, take an arbirtray vector $z\in X$ and let $b=Q(z,y)/Q(y,y)$. Then $Q(z-by,y)=0$, and thus,
        \begin{equation*}
            Q(x-ay,z)
            = Q ( x-ay,z-by) + b Q(x-ay,y)
            = 0.
        \end{equation*}
        This shows that $x-ay\in\rad Q$.

        \ref*{lem:hyperbolic_equal_case:radical} $\Rightarrow$ \ref*{lem:hyperbolic_equal_case:equality}: If $x-ay \in\rad Q$, then $Q(x-ay,x)=Q(x-ay,y)=0$,       
        which implies that the vector $(1,-a)^{\top}\in\R^2$ lies in the kernel of the matrix   
        \begin{equation*}
            S =
            \begin{pmatrix}
                Q(x,x) & Q(x,y) \\
                Q(y,x) & Q(y,y)
            \end{pmatrix}.
        \end{equation*}
        Hence, the matrix $S$ is singular and $Q(x,x)Q(y,y) - Q(x,y)^2 = \det S = 0$.
        
        For the final part of the statement, note that the equation $Q(x-ay,y)=0$ shows that $a=Q(x,y)/Q(y,y)$.
    \end{proof}

    \begin{lem}\label{lem:radical_subset_ker_xi}
        Let $Q$ be hyperbolic and $\xi:X\to \R$ be a linear form such that $Q$ is negative semidefinite on $\ker\xi$. Then $\rad Q \subseteq \ker \xi$.
    \end{lem}
    \begin{proof}
        Fix some vector $y\in X$ such that $Q(y,y)>0$, and thus, $\xi(y)\neq 0$. For a given vector $x\in\rad Q$, let $a=\xi(x)/\xi(y)$ and note that since $x-ay\in\ker \xi$, we have
        \begin{equation*}
            a^2 Q(y,y)
            = Q(x-ay,x-ay) \leq 0.
        \end{equation*}
        Hence, $a=0$ and we obtain that $x\in\ker \xi$.
    \end{proof}

    \section{Spectrally positive and balanced matrices}
    \label{sec:spec_pos_and_balanced}

    Now we take a closer look at the notions of spectral $k$-positivity and $k$-balance that we presented in the introduction. For our purposes, it will be practical to work with the strict versions of these terms.

    \begin{defi}
        Let $1\leq k\leq n-1$ and $Z\in\Sym^2(\R^n)$ be a symmetric matrix.
        \begin{enumerate}[label=\upshape(\roman*)]
            \item $Z$ is called \emph{strictly spectrally $k$-positive} if
        \begin{equation*}
       \sum_{i\in I} \lambda_i(Z) > 0
       \qquad\text{for any subset } I\subseteq [n], \, |I|=k.
        \end{equation*}
        \item $Z$ is called \emph{strictly spectrally $k$-balanced} if
        \begin{equation*}
       \sum_{i\in I} \lambda_i(Z) > \lambda_{\max}(Z)
       \qquad\text{for any subset } I\subseteq [n], \, |I|=k+1.
        \end{equation*}
        \end{enumerate}
    \end{defi}

    First, we note that strict spectral balance implies strict spectral positivity.   
    
    \begin{lem}\label{lem:spec_balanced_implies_positive}
        Let $1\leq j \leq n-1$ and $Z\in\Sym^2(\R^n)$ be strictly spectrally $(n-j)$-balanced. Then $Z$ is strictly spectrally $(n-j)$-positive.
    \end{lem}
    \begin{proof}
        Take a subset $I\subseteq [n]$ of size $|I|=n-j$ and some $i_0\in [n]\setminus I$. Then
        \begin{equation*}
            \sum_{i\in I}\lambda_i(Z)
            = \sum_{i\in I\cup \{i_0\} }\lambda_i(Z) - \lambda_{i_0}(Z)
            >  \lambda_{\max}(Z) - \lambda_{i_0}(Z)
            \geq 0,
        \end{equation*}
        and thus, the matrix $Z$ is strictly spectrally $(n-j)$-positive.
    \end{proof}

    Next, we characterize strict spectral positivity and balance in a more elementary way, namely in terms of orthonormal systems instead of eigenvalues.   
   The important consequence of this description is that spectral $(n-j)$-positivity and $(n-j)$-balance are each preserved under the operation $Z\mapsto Z^{\langle i\rangle}$ of removing the $i$-th row and column. 

    \begin{lem}\label{lem:spec_positive_iff}
        Let $1\leq j\leq n-1$ and $Z\in\Sym^2(\R^n)$. Then $Z$ is strictly spectrally $(n-j)$-positive if and only if for every orthonormal system $(u_1,\ldots,u_{n-j})$ of $\R^n$,
        \begin{equation}\label{eq:spec_positive_iff}
            \sum_{i=1}^{n-j} \pair{u_i}{Z u_i}
            > 0.
        \end{equation}
        In particular, if $1<j\leq n-1$ and $Z$ is strictly spectrally $(n-j)$-positive, then $Z^{\langle i\rangle}$ is is strictly spectrally $(n-j)$-positive for every $i\in [n]$.
    \end{lem}
    \begin{proof}
       Denote by $\lambda_1\leq \cdots \leq \lambda_n$ the eigenvalues of $Z$ (counted with multiplicity) and by $(v_1,\ldots,v_n)$ a corresponding orthonormal eigenbasis. If $Z$ satisfies \cref{eq:spec_positive_iff} for every orthonormal system $(u_1,\ldots,u_{n-j})$ of $\R^n$, we may simply choose $u_i=v_i$ for $i\in\{1,\ldots,n-j\}$, and then the left hand side of~\cref{eq:spec_positive_iff} amounts to $\lambda_1+\cdots+\lambda_{n-j}$.

       Conversely, suppose now that $\lambda_1+\cdots+\lambda_{n-j}\geq 0$ and take any orthonormal system $(u_1,\ldots,u_{n-j})$ of $\R^n$. Defining $\alpha_k = \sum_{i=1}^{n-j} \pair{u_i}{v_k}^2$ for $k\in [n]$, we have
       \begin{equation*}
           \sum_{i=1}^{n-j} \pair{u_i}{Z u_i}
           = \sum_{i=1}^{n-j} \sum_{k=1}^n \lambda_k \pair{u_i}{v_k}^2
           = \sum_{k=1}^n \alpha_k \lambda_k.
       \end{equation*}
       Observe that $0\leq \alpha_k \leq 1$ and that $\sum_{k=1}^{n-j}\alpha_k=n-j$. Due to these properties of the numbers $\alpha_k$ and the ordering of the eigenvalues $\lambda_k$,
       \begin{align*}
           \sum_{k=n-j+1}^n \alpha_k\lambda_k
           \geq \sum_{k=n-j+1}^n \alpha_k \lambda_{n-j}
           = \sum_{k=1}^{n-j} (1-\alpha_k)\lambda_{n-j}
           \geq \sum_{k=1}^{n-j} (1-\alpha_k)\lambda_k.
       \end{align*}
       Combining our computations above, we obtain that
       \begin{equation*}
           \sum_{i=1}^{n-j} \pair{u_i}{Z u_i}
           = \sum_{i=1}^{n-j} \alpha_i \lambda_i + \sum_{i=n-j+1}^n \alpha_i \lambda_i
           \geq \sum_{i=1}^{n-j} \alpha_i \lambda_i + \sum_{i=1}^{n-j} (1-\alpha_i)\lambda_i
           = \sum_{i=1}^{n-j} \lambda_i
           > 0,
       \end{equation*}
       and hence inequality \cref{eq:spec_positive_iff} is satisfied.
       
       The ``in particular'' part of the assertion follows immediately.
    \end{proof}

    In order to prove an analogous characterization for strictly $(n-j)$-balanced matrices, we require the following auxiliary lemma.

    \begin{lem}\label{lem:spec_balanced_sufficient}
        Let $2\leq j\leq n-1$ and $Z\in\Sym^2(\R^n)$. Then $Z$ is strictly spectrally $(n-j)$-balanced if and only if
        \begin{equation}\label{eq:spec_balanced_sufficient}
       \sum_{i\in I} \lambda_i(Z) > \max_{i\in [n]\setminus I} \lambda_i(Z)
       \qquad\text{for any subset } I\subseteq [n], \, |I|=n-j+1.
        \end{equation}
    \end{lem}
    \begin{proof}
        If $Z$ is strictly spectrally $(n-j)$-balanced, then $Z$ also satisfies \cref{eq:spec_balanced_sufficient}, since $\lambda_(Z) \geq \max_{i\in [n]\setminus I} \lambda_i(Z)$.

        Conversely, suppose that $Z$ satisfies \cref{eq:spec_balanced_sufficient}.
        Let $I\subseteq [n]$ be a subset of size $|I|=n-j+1$ and $i_0\in [n]$ be such that $\lambda_{i_0}(Z)=\lambda_{\max}(Z)$. If $i_0\in [n]\setminus I$, then
        \begin{equation*}
            \sum_{i\in I}\lambda_i(Z)
            > \max_{i\in [n]\setminus I} \lambda_i(Z)
            \geq \lambda_{i_0}(Z)
        \end{equation*}
        by our assumption on $Z$. If $i_0\in I$, by \cref{lem:spec_balanced_implies_positive},
        \begin{equation*}
            \sum_{i\in I}\lambda_i(Z)
            = \sum_{i\in I\setminus\{i_0\}} \lambda_i(Z)  + \lambda_{i_0} (Z)
            > \lambda_{i_0}(Z).
        \end{equation*}
        Via this case distinction, we obtain that $Z$ is strictly spectrally $(n-j)$-balanced. 
    \end{proof}

    \begin{lem}\label{lem:spec_balanced_iff}
        Let $2\leq j \leq n-1$ and $Z\in\Sym^2(\R^n)$. Then $Z$ is strictly spectrally $(n-j)$-balanced if and only if for any orthonormal system $(u_0,u_1,\ldots,u_{n-j+1})$ of $\R^n$,
        \begin{equation}\label{eq:spec_balanced_iff}
            \sum_{i=1}^{n-j+1} \pair{u_i}{Z u_i}
            > \pair{u_0}{Z u_0}.
        \end{equation}
        In particular, if $2<j\leq n-1$ and $Z$ is strictly spectrally $(n-j)$-balanced, then $Z^{\langle i\rangle}$ is strictly spectrally $(n-j)$-balanced for each $i\in [n]$.
    \end{lem}
    \begin{proof}
        Let $(v_1,\ldots,v_n)$ be an orthonormal basis of eigenvectors of $Z$ and denote the corresponding eigenvalues by $\lambda_1(Z),\ldots,\lambda_n(Z)$. Suppose first that $Z$ satisfies \cref{eq:spec_balanced_iff} for every orthonormal system $(u_0,u_1,\ldots,u_{n-j+1})$ of $\R^n$ and take a subset $I\subseteq [n]$, $|I|=n-j+1$. Then we may choose $(u_1,\ldots,u_{n-j+1})$ to be the family $(v_i)_{i\in I}$ and $u_0$ to be the eigenvector corresponding to the eigenvalue $\max_{i\in [n]\setminus I}\lambda_i(Z)$. Inequality~\cref{eq:spec_balanced_iff} then reads as
        \begin{equation*}
            \sum_{i\in I} \lambda_i(Z)
            \geq \max_{i\in [n]\setminus I}\lambda_i(Z).
        \end{equation*}
        \cref{lem:spec_balanced_sufficient} yields that $Z$ is strictly spectrally $(n-j)$-balanced.

        Conversely, suppose that $Z$ is strictly spectrally $(n-j)$-balanced and consider the matrix $\tilde{Z}\in\Sym^2(\R^n)$, defined as
        \begin{equation*}
            \tilde{Z} = Z - \frac{\lambda_{\max}(Z)}{n-j+1} I_n.
        \end{equation*}
        \cref{lem:spec_balanced_implies_positive} implies that for any subset $I\subseteq [n]$ of size $|I|=n-j+1$,
        \begin{equation*}
            \sum_{i\in I} \lambda_i(\tilde{Z})
            = \sum_{i\in I} \bigg(\lambda_i(Z) - \frac{\lambda_{\max}(Z)}{n-j+1}\bigg)
            = \sum_{i\in I} \lambda_i(Z) - \lambda_{\max}(Z)
            > 0.
        \end{equation*}
        Hence, the matrix $\tilde{Z}$ is strictly spectrally $(n-j+1)$-positive. Next, take an orthonormal system $(u_0,u_1,\ldots,u_{n-j+1})$ of $\R^n$. \cref{lem:spec_positive_iff}, applied to $\tilde{Z}$ gives
        \begin{equation*}
            \sum_{i=1}^{n-j+1} \pair{u_i}{Z u_i}
            = \sum_{i=1}^{n-j+1} \pair{u_i}{\tilde{Z} u_i} + \lambda_{\max}(Z)
            > \lambda_{\max}(Z)
            \geq \pair{u_0}{Z u_0},
        \end{equation*}
        and hence, inequality~\cref{eq:spec_balanced_iff} is satisfied.

        The ``in particular'' part of the assertion follows immediately.
    \end{proof}

    \section{Mixed discriminant inequalities}
    \label{sec:mixed_discriminants}

    In this section, we establish the mixed discriminant inequality of~\cref{thmintro:mixed_discr_ineq}. First, we recall the definition and basic properties of mixed discriminants and discuss the notions of spectrally positive and balanced matrices presented in the introduction.

    \subsection{Preparations}

    The \emph{mixed discriminant} of $n$ matrices $A_1,\ldots,A_n\in\Sym^2(\R^n)$, where we denote $A_k=(a_{ij}^k)_{i,j\in [n]}$, is defined as
    \begin{equation}\label{eq:mixed_discriminant_defi}
        \D(A_1,\ldots,A_n)
        = \frac{1}{n!}\sum_{\sigma\in S_n} \left| \begin{array}{ccc}
            a_{11}^{\sigma(1)} & \cdots & a_{1n}^{\sigma(n)} \\
            \vdots & \ddots & \vdots \\
            a_{n1}^{\sigma(1)} & \cdots & a_{nn}^{\sigma(n)}
        \end{array}\right|,
    \end{equation}
    where $S_n$ is the symmetric group of all permutations of $[n]=\{1,\ldots,n\}$. The mixed discriminant is symmetric in its arguments, linear in each argument, and it polarizes the determinant, understood as an $n$-homogeneous polynomial on $\Sym^2(\R^n)$. 
    For all $m\in\N$, $A_1,\ldots,A_m\in\Sym^2(\R^n)$, and $\lambda_1,\ldots,\lambda_m\geq 0$,
    \begin{equation*}
        \det(\lambda_1A_1 + \cdots + \lambda_mA_m)
        = \sum_{i_1,\ldots,i_n\in [m]} \D(A_{i_1},\ldots,A_{i_n}) \lambda_{i_1}\cdots \lambda_{i_n}.
    \end{equation*}
    The following lemma collects further important properties of the mixed discriminant that we will use repeatedly (see~\cite{Shenfeld2019}*{Lemma~2.6}). Here and in the following, we denote by $\Sym^2_{++}(\R^n)$ the set of positive definite symmetric $n$-dimensional matrices. Moreover, given a matrix $A\in\Sym^2(\R^n)$ and an index $i\in [n]$, we denote by $A^{\langle i\rangle} \in \Sym^2(\R^{n-1})$ the matrix obtained from $A$ by removing the $i$-th row and $i$-th column. 

    \begin{lem}[{\cite{Shenfeld2019}}]\label{lem:mixed_discr_properties}
        Let $A_1,\ldots,A_n\in\Sym^2(\R^n)$ and let $B\in\Sym^2(\R^n)$ be invertible. 
        \begin{enumerate}[label=\upshape(\roman*)]
            \item $\D(A,\ldots,A)=\det(A)$.
            \item $\D(B^{-1}A_1B,\ldots,B^{-1}A_nB)=\D(A_1,\ldots,A_n)$.
            \item \label{lem:mixed_discr_properties:positivity}
            $\D(A_1,\ldots,A_n)>0$ whenever $A_1,\ldots,A_{n}\in\Sym^2_{++}(\R^n)$.
            \item \label{lem:mixed_discr_properties:restr}
            $\D(A_1,\ldots,A_{n-1},e_ie_i^{\ast}) = \frac{1}{n}\D(A_1^{\langle i\rangle},\ldots,A_{n-1}^{\langle i\rangle})$.
        \end{enumerate}
    \end{lem}

    To deal with mixed discriminants of the form $\D(A,B,C_1,\ldots,C_{j-2},Z,I_n^{[n-j-1]})$, the next lemma is especially useful.
    Given a matrix $A\in\Sym^2(\R^n)$ and a subset $I\subseteq [n]$ of size $|I|=n-j$, we denote by $A^{\langle I \rangle} \in \Sym^2(\R^{j})$ the matrix obtained from $A$ by removing all rows and columns with index $i\in I$.

    \begin{lem}\label{lem:mixed_discr_identity_positivity}
        Let $1\leq j\leq n-1$ and $Z=\mathrm{diag}(\lambda_1,\ldots,\lambda_n)$, where $\lambda_1,\ldots,\lambda_n\in\R$. Then for all $A_1,\ldots,A_j\in\Sym^2(\R^n)$,
        \begin{equation}\label{eq:mixed_discr_identity_positivity}
            \D(A_1,\ldots,A_j,Z,I_n^{[n-k-1]} )
            =  \frac{j!(n-j-1)!}{n!} \sum_{\substack{I\subseteq [n] \\ |I|=n-j}} \D(A_1^{\langle I\rangle},\ldots,A_j^{\langle I\rangle}) \sum_{i\in I} \lambda_i.
        \end{equation}
    \end{lem}
    \begin{proof}
        First, observe according to \cref{lem:mixed_discr_properties}~\ref{lem:mixed_discr_properties:restr}, we have
        \begin{align*}
            \D(A^{[j]},Z,I_n^{[n-j-1]} )
            = \frac{1}{n} \sum_{i=1}^n \lambda_i \, \D((A^{\langle i\rangle})^{[j]},I_{n-1}^{[n-j-1]}).
        \end{align*}
        Next, note that from definition~\cref{eq:mixed_discriminant_defi} it is not hard to see that
        \begin{align*}
            \D((A^{\langle i\rangle})^{[j]},I_{n-1}^{[n-j-1]})
            = \frac{1}{\binom{n-1}{j}} \sum_{\substack{J\subseteq [n]\setminus \{i\} \\ |J|=j }} \det ((a_{k\ell})_{k,\ell\in J})
            =\frac{1}{\binom{n-1}{j}} \sum_{\substack{i\in I\subseteq [n]\\|I|=n-j}} \det (A^{\langle I\rangle } ).
        \end{align*}
        Combining our computations yields
        \begin{align*}
            \D(A^{[j]},Z,I_n^{[n-j-1]} )
            = \frac{1}{n\binom{n-1}{j}} \sum_{i=1}^n \lambda_i   \, \sum_{\substack{i\in I\subseteq [n]\\|I|=n-j}} \det (A^{\langle I\rangle })
            = \frac{1}{n\binom{n-1}{j}} \sum_{\substack{ I\subseteq [n] \\ |I|=n-j }} \det (A^{\langle I\rangle }) \sum_{i\in I} \lambda_i,
        \end{align*}
        which is identity~\cref{eq:mixed_discr_identity_positivity} in the instance where $A_1=\cdots=A_j=A$. By polarizing both sides of the equation, we obtain the desired identity in its full generality.        
    \end{proof}

    It turns out that $Z$ being spectrally $(n-j)$-positive is precisely the property that makes mixed discriminants of the form $\D(A_1,\ldots,A_j,Z,I_n^{[n-j-1]})$ nonnegative. This simple fact might also be of broader interest in the pursuit of more general mixed discriminant inequalities. When one asks for which choices of $Z\in\Sym^2(\R^n)$, the matrix functional $\D({{}\cdot{}}^{[j]},Z,I_n^{[n-j-1]})$ is log-concave, for instance, it is natural to restrict to those that make the functional nonnegative.       
    
    \begin{lem}\label{mixed_discr_pos_iff} Let $1\leq j\leq n-1$ and $Z\in\Sym^2(\R^n)$. Then the following are equivalent.
        \begin{enumerate}[label=\upshape(\alph*)]
            \item \label{mixed_discr_pos_iff:cond}
            $Z$ is strictly spectrally $(n-j)$-positive.
            \item \label{mixed_discr_pos_iff:pos}
            $\D(A^{[j]},Z,I_n^{[n-j-1]})> 0$ for all $A \in\Sym^2_{++}(\R^n)$.
            \item \label{mixed_discr_pos_iff:pos_mixed}
            $\D(A_1,\ldots,A_j,Z,I_n^{[n-j-1]})> 0$ for all $A_1,\ldots,A_j\in\Sym^2_{++}(\R^n)$.
        \end{enumerate}
    \end{lem}
    \begin{proof}        
        First, we choose an orthonoromal coordinate system $(v_1,\ldots,v_n)$ consisting of eigenvectors of $Z$, so that $Z=\diag(\lambda_1,\ldots,\lambda_n)$.

        \ref*{mixed_discr_pos_iff:cond}$\Rightarrow$\ref*{mixed_discr_pos_iff:pos_mixed}: If $Z$ is strictly spectrally $(n-j)$-positive, then $\sum_{i\in I}\lambda_i > 0$ for every set $I\subseteq [n]$ of size $|I|=n-j$. Moreover, $\D(A_1^{\langle I\rangle},\ldots,A_j^{\langle I\rangle})> 0$ for all $A_1,\ldots,A_j\in\Sym^2_{++}(\R^n)$, and thus, identity~\cref{eq:mixed_discr_identity_positivity} yields $\D(A_1,\ldots,A_j,Z,I_n^{[n-j-1]})> 0$.

        \ref*{mixed_discr_pos_iff:pos_mixed}$\Rightarrow$\ref*{mixed_discr_pos_iff:pos}: Follows immediatley by choosing $A_1=\cdots=A_j=A$.

        \ref*{mixed_discr_pos_iff:pos}$\Rightarrow$\ref*{mixed_discr_pos_iff:cond}: Let $I\subseteq [n]$ be an arbitrary set of indices of size $|I|=n-j$ and take
        $A=\sum_{i\in [n]\setminus I} v_i v_i^{\ast} \in\Sym^2_{+}(\R^n)$ to be the orthogonal projection onto the $j$-dimensional subspace spanned by the family $\{v_i \colon i\in [n]\setminus I\}$. Observe that for any given set $J\subseteq [n]$ of size $|J|=n-j$, we have $\det(A^{\langle J\rangle})=1$ if $I=J$ and $\det(A^{\langle J\rangle})=0$ otherwise. Identity~\cref{eq:mixed_discr_identity_positivity} yields
        \begin{equation*}
            \sum_{i\in I}\lambda_i
            = \frac{j!(n-j-1)!}{n!} \D(A^{[j]},Z,I_n^{[n-j-1]}) > 0,
        \end{equation*}
        and thus, $Z$ is strictly spectrally $(n-j)$-positive.
    \end{proof}

    \subsection{Proof of \texorpdfstring{\cref{thmintro:mixed_discr_ineq}}{Theorem~C}}

    We are now ready to prove \cref{thmintro:mixed_discr_ineq}. From our discussion in \cref{sec:hyperbolicity}, we need to show that the quadratic form $\D({}\cdot{},{}\cdot{},\bm{C},Z,I_n^{[n-j-1]})$ is hyperbolic whenever the matrix~$Z$ is spectrally $(n-j)$-balanced.
    We will obtain the desired theorem by approximation from the following ``strict'' formulation.
    
    \begin{thm}\label{thm:mixed_discr_strictly}
        Let $2\leq j\leq n-1$ and $Z\in\Sym^2(\R^n)$ be strictly spectrally $(n-j)$-balanced. Then for all matrices $\bm{C}=(C_1,\ldots,C_{j-2}) \subseteq \Sym^2_{++}(\R^n)$, the quadratic form $\D({}\cdot{},{}\cdot{},\bm{C},Z,I_n^{[n-j-1]})$ is nondegenerate and hyperbolic on $\Sym^2(\R^n)$.
    \end{thm}

    We will show \cref{thm:mixed_discr_strictly} by an induction on $j\geq 2$.    
    The base case, where $j=2$ and no reference matrices are present, the condition on $Z$ to be spectrally $(n-2)$-balanced is precisely what is needed for the quadratic form $\D({}\cdot{},{}\cdot{},Z,I_n^{[n-3]})$ to be negative semidefinite on the hyperplane of matrices with vanishing trace. This is discussed in the following proposition; its proof is based on an elementary yet clever way to express the mixed discriminant.

    \begin{prop}\label{prop:mixed_discr_base_case}
        Let $n\geq 3$ and $Z\in\Sym^2(\R^n)$ be strictly spectrally $(n-2)$-balanced. Then the quadratic form $\D({}\cdot{},{}\cdot{},Z,I_n^{[n-3]})$ is nondegenerate and hyperbolic on $\Sym^2(\R^n)$.
    \end{prop}
    \begin{proof}
        First, choose a coordinate system consisting of eigenvectors of $Z$, so that we have $Z=\diag(\lambda_1,\ldots,\lambda_n)$, where $\lambda_1,\ldots,\lambda_n\in\R$.
        Then, take a matrix $A\in\Sym^2(\R^n)$, and note that by identity \cref{eq:mixed_discr_identity_positivity}, we have
        \begin{equation*}
            \D(A,A,Z,I_n^{[n-3]})
            = \frac{1}{n(n-1)(n-2)} \sum_{j,k \in [n]} (a_{jj}a_{kk}-a_{jk}^2)(\tr(Z) - \lambda_j - \lambda_k),
        \end{equation*}
        where we also used the fact that $a_{jj}a_{kk}-a_{jk}^2 =0$ whenever $j=k$. Observe that for the part containing $a_{jj}a_{kk}$, we may exploit the symmetry of the expression so that
        \begin{equation*}
            \sum_{j,k \in [n]} a_{jj}a_{kk}(\tr(Z) - \lambda_j - \lambda_k)
            = \sum_{j,k \in [n]} a_{jj}a_{kk}(\tr(Z) - 2\lambda_j)
            = \tr(A) \sum_{j\in [n]} a_{jj}(\tr(Z) - 2\lambda_j).
        \end{equation*}
        Hence, whenever $\tr(A)=0$, we have
        \begin{equation*}
            \D(A,A,Z,I_n^{[n-3]})
            =  - \frac{1}{n(n-1)(n-2)} \sum_{j,k \in [n]} a_{jk}^2 (\tr(Z) - \lambda_j - \lambda_k).
        \end{equation*}
        The assumption that $Z$ is strictly $(n-2)$-balanced implies that $\tr(Z)> 2 \lambda_{\max}(Z)$. In particular, $\tr(Z) > \lambda_j + \lambda_k$ for all $j,k\in [n]$. 
        This shows that the quadratic form $\D({}\cdot{},{}\cdot{},Z,I_n^{[n-3]})$ is negative definite on $\ker\xi$, where $\xi=\tr:\Sym^2(\R^n)\to\R$. 
        Therefore, by \cref{lem:hyperbolic_iff,lem:radical_subset_ker_xi} the quadratic form $Q = \D({}\cdot{},{}\cdot{},Z,I_n^{[n-3]})$ is hyperbolic on the vector space $\Sym^2(\R^n)$ and $\rad Q=\{0\}$.
    \end{proof}

    Next we carry out the induction step for~\cref{thm:mixed_discr_strictly}; the argument is inspired by Alexandrov's proof of the Alexandrov--Fenchel inequality~\cite{Alexandrov1938}. For every strictly $(n-j)$-balanced matrix $Z$, we show that the quadratic form $\D({}\cdot{},{}\cdot{},\bm{C},Z,I_n^{[n-j-1]})$ is nondegenerate, using the induction hypothesis. Then, by ``connecting'' $Z$ to the identity matrix $I_n$, we show that the hyperbolicity is preserved, employing the following homotopy argument which is tailored to our purposes.

    \begin{lem}\label{lem:homotopy_argument}
        Let $\mathcal{S}$ be a connected set of nondegenerate quadratic forms on a finite-dimensional real vector space $X$ and suppose there exists a vector $y\in X$ such that $Q(y,y)>0$ for all $Q\in\mathcal{S}$. Then all or none of the quadratic forms in $\mathcal{S}$ are hyperbolic.
    \end{lem}
    \begin{proof}
        First, note that the space $X^{\ast}\otimes X^{\ast}$ of bilinear forms on $X$ naturally carries the topology of uniform convergence on bounded subsets of $X$.
        Let $\mathcal{S}_0 \subseteq \mathcal{S}$ denote the set of hyperbolic quadratic forms in $\mathcal{S}$. Expressing this set as 
        \begin{equation*}
           \mathcal{S}_0
           = \bigcap_{x\in X} \{Q\in\mathcal{S} : Q(x,y)^2 \geq Q(x,x)Q(y,y)\}
        \end{equation*}
        shows that $\mathcal{S}_0$ is closed in $\mathcal{S}$. Next, we endow the vector space $X$ with some norm $\lVert {}\cdot{} \rVert$ and take some hyperbolic quadratic from $Q_0\in\mathcal{S}_0$. By its nondegenerate, there exists some linear form $\xi:X\to\R$ such that $Q_0$ is negative definite on $\ker\xi$. Observe that the map
        \begin{equation*}
            f:\mathcal{S}\to\R:f(Q) = \sup \{ Q(x,x) : x\in \ker\xi \text{ and } \lVert x\rVert =1 \}
        \end{equation*}
        is continuous on $\mathcal{S}$ and $f(Q_0)<0$ by compactness. Hence, $f<0$ in some neighborhood of $Q_0$, which implies that $\mathcal{S}_0$ is open in $\mathcal{S}$.
        Since $\mathcal{S}$ is connected and $\mathcal{S}_0\subseteq\mathcal{S}$ is both open and closed in $\mathcal{S}$ we must have $\mathcal{S}_0=\emptyset$ or $\mathcal{S}_0=\mathcal{S}$.
    \end{proof}

    \begin{proof}[Proof of \cref{thm:mixed_discr_strictly}]
        We prove the statement by induction on $j \geq 2$. The base case of $j=2$ is covered by \cref{prop:mixed_discr_base_case}.
        For the induction step, let $j>2$, assume that $Z$ is strictly $(n-j)$-balanced and take $C_1,\ldots,C_{j-2} \in \Sym^2_{++}(\R^n)$. Moreover we write $C=C_1$ and $\bm{C'}=(C_2,\ldots,C_{j-2})$ and choose a coordinate system of $\R^n$ such that $C=\diag(c_1,\ldots,c_n)$.

        According to \cref{lem:spec_balanced_iff}, for each $i\in [n]$, the matrix $Z^{\langle i\rangle} \in\Sym^2(\R^{n-1})$ is strictly spectrally $(n-j)$-balanced. Since $n-j=(n-1)-(j-1)$, the induction hypothesis gives that the quadratic form $\D({}\cdot{},{}\cdot{},{\bm{C'}}^{\langle i\rangle},Z^{\langle i\rangle},I_{n-1}^{[n-j-1]})$ is nondegenerate and hyperbolic on $\Sym^2(\R^{n-1})$. Moreover, it is strictly positive on positive definite matrices by \cref{lem:spec_balanced_implies_positive}.
    	Consequently, for any matrix $A\in \Sym^2(\R^n)$,
        \begin{align*}
            \D(A,A,C,\bm{C'},Z,I_n^{[n-k-1]})
            &= \frac{1}{n} \sum_{i=1}^n c_i \D(A^{\langle i\rangle},A^{\langle i\rangle},{\bm{C'}}^{\langle i\rangle},Z^{\langle i\rangle},I_{n-1}^{[n-j-1]}) \\
            &\leq \frac{1}{n} \sum_{i=1}^n c_i \frac{\D(A^{\langle i\rangle},C^{\langle i\rangle},{\bm{C'}}^{\langle i\rangle},Z^{\langle i\rangle},I_{n-1}^{[n-j-1]})^2}{\D(C^{\langle i\rangle},C^{\langle i\rangle},{\bm{C'}}^{\langle i\rangle},Z^{\langle i\rangle},I_{n-1}^{[n-j-1]})}.
        \end{align*}
        
        Suppose now that $A\in\Sym^2(\R^n)$ lies in the radical of $\D({}\cdot{},{}\cdot{},\bm{C},Z,I_n^{[n-j-1]})$. Then the left hand side of the inequality above vanishes and
        \begin{equation*}
            \D(A^{\langle i\rangle},C^{\langle i\rangle},{\bm{C'}}^{\langle i\rangle},Z^{\langle i\rangle},I_{n-1}^{[n-j-1]})
            = n \D(A,e_ie_i^{\ast},C,\bm{C'},Z,I_n^{[n-j-1]})
            = 0,
        \end{equation*}
        so the right hand side vanishes as well. Consequently, for all $i\in [n]$, we have
        \begin{equation*}
            \D(A^{\langle i\rangle},A^{\langle i\rangle},{\bm{C'}}^{\langle i\rangle},Z^{\langle i\rangle},I_{n-1}^{[n-k-1]})
            = \frac{\D(A^{\langle i\rangle},C^{\langle i\rangle},{\bm{C'}}^{\langle i\rangle},Z^{\langle i\rangle},I_{n-1}^{[n-k-1]})^2}{\D(C^{\langle i\rangle},C^{\langle i\rangle},{\bm{C'}}^{\langle i\rangle},Z^{\langle i\rangle},I_{n-1}^{[n-k-1]})}
            = 0.
        \end{equation*}
        \cref{lem:hyperbolic_equal_case} implies that $A^{\langle i\rangle}$ lies in the radical of 
        $\D({}\cdot{},{}\cdot{},{\bm{C'}}^{\langle i\rangle},Z^{\langle i\rangle},I_{n-1}^{[n-j-1]})$ for each $i\in[n]$. By the nondegeneracy part of the induction hypothesis, $A^{\langle i\rangle}=0$ for $i\in[n]$, and thus, $A=0$. This shows that the quadratic form $\D({}\cdot{},{}\cdot{},\bm{C},Z,I_n^{[n-j-1]})$ is nondegenerate.

        For the final part, fix the reference bodies $\bm{C}=(C_1,\ldots,C_{j-2}) \subseteq \K^n$ and denote by $\mathcal{S}$ the convex cone of quadratic forms $Q_Z=\D({}\cdot{},{}\cdot{},\bm{C},Z,I_n^{[n-j-1]})$ on $\Sym^2(\R^n)$, where $Z$ is strictly spectrally $(n-j)$-balanced. By our previous argument, each quadratic form $Q\in\mathcal{S}$ is nondegenerate and by \cref{lem:spec_balanced_implies_positive}, we have $Q(I_n,I_n)>0$ for all $Q\in\mathcal{S}$. Moreover, for the choice $Z=I_n$, Alexandrov's mixed discriminant inequality \cref{eq:Alexandrov_matrix_ineq} and \cref{lem:hyperbolic_on_cone} assert that $Q_{I_n}$ is hyperbolic. Hence, \cref{lem:homotopy_argument} yields that each quadratic form $Q\in \mathcal{S}$ is hyperbolic. This completes the induction step.
    \end{proof}

    An approximation argument yields that if $Z$ is spectrally $(n-j)$-balanced and the reference matrices are positive semidefinite, then the quadratic form $\D({}\cdot{},{}\cdot{},\bm{C},Z,I_n^{[n-j-1]})$ is possibly degenerate but still hyperbolic. In particular, we obtain~\cref{thmintro:mixed_discr_ineq}, which we state here in a slightly stronger form.

    \begin{thm}\label{thm:Alexandrov_type_ineq}
        Let $2\leq j\leq n-1$ and $Z\in\Sym^2(\R^n)$ be spectrally $(n-j)$-balanced. Then for all $A\in\Sym^2(\R^n)$, $B\in\Sym^2_{+}(\R^{n})$ and $\bm{C}=(C_1,\ldots,C_{j-2}) \subseteq \Sym^2_{+}(\R^{n})$,
        \begin{align}\label{eq:Alexandrov_type_ineq}
        \D(A,B,\bm{C},Z,I_n^{[n-j-1]})^2 
         \geq \D(A,A,\bm{C},Z,I_n^{[n-j-1]}) \D(B,B,\bm{C},Z,I_n^{[n-j-1]}).
        \end{align}
    \end{thm}
    \begin{proof}
        For every $\varepsilon>0$, the matrix
        $Z_{\varepsilon} = Z + \varepsilon I_n$ is strictly spectrally $(n-j)$-balanced and        
        $\bm{C}_{\varepsilon}=(C_1 + \varepsilon I_n,\ldots,C_{j-2}+\varepsilon I_n) \subseteq \Sym^2_{++}(\R^n)$. By \cref{thm:mixed_discr_strictly} and \cref{lem:spec_balanced_implies_positive} the quadratic form $\D({}\cdot{},{}\cdot{},\bm{C}_{\varepsilon},Z_{\varepsilon},I_n^{[n-j-1]})$ is hyperbolic on $\Sym^2(\R^n)$ and nonnegative on positive semidefinite matrices. Hence,       
        \begin{align*}
        \D(A,B,\bm{C}_{\varepsilon},Z_{\varepsilon},I_n^{[n-j-1]})^2 
         \geq \D(A,A,\bm{C}_{\varepsilon},Z_{\varepsilon},I_n^{[n-j-1]}) \D(B,B,\bm{C}_{\varepsilon},Z_{\varepsilon},I_n^{[n-j-1]}),
        \end{align*}
        and by passing to the limit $\varepsilon \to 0$, we obtain the desired inequality.
    \end{proof}

    \subsection{Consequences}

    To conclude this section, we note three further mixed discriminant inequalities involving a spectrally $(n-j)$-balanced matrix $Z\in\Sym^2(\R^n)$ arising from the Alexandrov type inequality~\cref{eq:Alexandrov_type_ineq}. The proofs follow standard algebraic manipulations of log-concave sequences. We omit them here, but carry out the completely analogous arguments later in the proofs of \cref{thm:Minkowsi-ineq_zeta,thm:B-M-ineq_zeta:additive,thm:B-M-ineq_zeta}. In the following, we assume $A,B\in\Sym^2_+(\R^n)$, $0 < i< j$, and $\lambda\in (0,1)$.
    
    \begin{equation}
        \D(A^{[i]},B^{[j-i]},Z,I_n^{[n-j-1]})^{j}
        \geq \D(A^{[j]},Z,I_n^{[n-j-1]})^{i}\, \D(B^{[j]},Z,I_n^{[n-j-1]})^{j-i}
    \end{equation}
    \begin{align}
        \begin{split}
            &\D(((1-\lambda)A+\lambda B)^{[j]},Z,I_n^{[n-j-1]})^{\frac{1}{j}} \\
            &\qquad \geq (1-\lambda) \D(A^{[j]},Z,I_n^{[n-j-1]})^{\frac{1}{j}} + \lambda \D(B^{[j]},Z,I_n^{[n-j-1]})^{\frac{1}{j}}.
        \end{split}            
    \end{align}
    \begin{equation}
        \D(((1-\lambda)A+\lambda B)^{[j]},Z,I_n^{[n-j-1]}) 
        \geq \D(A^{[j]},Z,I_n^{[n-j-1]})^{1-\lambda} \, \D(B^{[j]},Z,I_n^{[n-j-1]})^{\lambda}.
    \end{equation}

    \section{Mixed volume inequalities}
    \label{sec:mixed_vols}

    This section deals with several mixed volume inequalities with some function $\zeta$ taking the role of (the support function of) a convex body; in particular, we prove \cref{thmintro:Alexandrov-Fenchel,thmintro:Brunn-Minkowski}. Before that, we recall some important facts about mixed volumes; as a general reference, we recommend~\cite{Schneider2014} and \cite{Shenfeld2019}.

    \subsection{Preparations}
    \label{sec:mixed_vols:preps}

    We recall that for a function $f\in C^2(\S^{n-1})$ and $u\in\S^{n-1}$, we denote by $D^2f(u)$ the Hessian of its one-homogeneous extension to $\R^n\setminus \{0\}$ at the point $u$. By homogeneity, this matrix satisfies $D^2f(u) u = 0$, and may thus be interpreted as a self-adjoint linear endomorphism of $u^{\perp}$, that is, $D^2f(u)\in\Sym^2(u^{\perp})$.
    
    Moreover, $f\in C^2(\S^{n-1})$ is a support function of some convex body if and only if $D^2f(u)$ is positive semidefinite at every point $u\in\S^{n-1}$. If $K\in\K^n$ is a convex body of class $C^2_+$ (that is, its boundary is a $C^2$ submanifold with positive Gauss curvature), then $D^2h_K(u)$ is positive definite on $\S^{n-1}$. For instance, for $K=\BB$, we have $D^2h_{\BB}(u) = I_{u^\perp}$, where here and in the following $I_{u^\perp}$ denotes the identity map on $u^\perp \subseteq \R^n$. 

    Given a function $f\in C^2(\S^{n-1})$, by compactness of the unit sphere, there exists $r\geq 0$ such that $D^2f(u) + r I_{u^\perp}$ is positive definite on $\S^{n-1}$. Hence, every $C^2(\S^{n-1})$ function can be written as the difference of two support functions, allowing us to extend the mixed volume to a multilinear form on the space $C^2(\S^{n-1})$.
    
    The following classical formula expresses the mixed volume of $n$~convex bodies in $\R^n$ as an integral involving a mixed discriminant:
    \begin{equation}\label{eq:mixed_volume_integral}
        \V(K_1,\ldots,K_n)
        = \frac{1}{n} \int_{\S^{n-1}} h_{K_1} \, 
        \D(D^2h_{K_2},\ldots,D^2h_{K_n})\, d\omega,
    \end{equation}
    where here and in the following, $\omega$ denotes spherical Lebesgue measure.
    This identity also naturally extends to $C^2(\S^{n-1})$ functions. 
    The classical Alexandrov--Fenchel inequality~\cref{eq:Alexandrov-Fenchel_ineq} can thus be interpreted as the quadratic form $\V({}\cdot{},{}\cdot{},M_1,\ldots,M_{n-2})$ being hyperbolic on $C^{2}(\S^{n-1})$ (see also \cref{exl:A-F_ineq_hyperbolic}). Since we may express this quadratic form in terms of a differential operator on $\S^{n-1}$ via~\cref{eq:mixed_volume_integral}, the hyperbolicity is equivalent to a spectral gap.     
    Shenfeld and van Handel~\cite{Shenfeld2019} gave a modern and concise proof of the Alexandrov--Fenchel inequality based on such a spectral gap argument,
    and in here, we follow their approach to prove~\cref{thmintro:Alexandrov-Fenchel}.

    \subsection{Proof of \texorpdfstring{\cref{thmintro:Alexandrov-Fenchel}}{Theorem~B}}

    To this end, let $\zeta\in C^2(\S^{n-1})$ be such that for every $u\in\S^{n-1}$, the matrix $D^2\zeta(u)$ is positive definite or strictly spectrally $(n-j)$-balanced and let $M_1,\ldots,M_{j-2}\in\K^n$ be of class $C^{\infty}_+$; moreover, assume that $M_1$ contains the origin in its interior and write $h=h_{M_1}$. Define the a differential operator $\mathscr{A}$ and Borel measure $\mu$ on $\S^{n-1}$ as
    \begin{align}\label{eq:differential_operator}
        \begin{split}
        \mathscr{A}f
        &= h \, \frac{\D(D^2f,D^2h_{M_1},\ldots,D^2h_{M_{j-2}},D^2\zeta,I_{u^{\perp}}^{[n-j-1]}) }{\D(D^2h,D^2h_{M_1},\ldots,D^2h_{M_{j-2}},D^2\zeta,I_{u^{\perp}}^{[n-j-1]})}, \\
        d\mu
        &= \frac{1}{n} \frac{1}{h} \, \D(D^2h,D^2h_{M_1},\ldots,D^2h_{M_{j-2}},D^2\zeta,I_{u^{\perp}}^{[n-j-1]})\, d\omega.
        \end{split}
    \end{align}

    This differential operator $\mathscr{A}$ is very convenient for our purposes; it has several desirable properties that we summarize in the proposition below.
    
    \begin{prop}\label{prop:differential_operator}
        $\mathscr{A}$ is a uniformly elliptic operator that extends to a densely defined self-adjoint operator on $L^2(\mu)$ with discrete spectrum and $C^\infty(\S^{n-1})$ eigenfunctions. Its largest eigenvalue is~$1$, and the corresponding eigenspace is spanned by~$h$. Moreover, for all $f,g\in C^2(\S^{n-1})$,
        \begin{equation}\label{eq:mixed_vol_diff_op}
        \V(f,g,M_1,\ldots,M_{j-2},\zeta,{\BB}^{[n-j-1]})
        = \pair{f}{\mathscr{A}g}_{L^2(\mu)}
        = \pair{g}{\mathscr{A}f}_{L^2(\mu)}.
    \end{equation}
    \end{prop}
    \begin{proof}
        \cref{lem:spec_balanced_implies_positive,lem:spec_positive_iff} imply that
        $\D(A,D^2h_{M_1},\ldots,D^2h_{M_{j-2}},D^2\zeta,I_{u^{\perp}}^{[n-j-1]})>0$ for every $u\in\S^{n-1}$ and matrix $A\in \Sym^2_{++}(u^\perp)$. Hence, by compactness of the unit sphere, the differential operator $\mathscr{A}$ is uniformly elliptic.
        The identity \cref{eq:mixed_vol_diff_op} is due to \cref{eq:mixed_volume_integral}; in particular, the operator $\mathscr{A}$ defines a symmetric quadratic form.

        Standard elliptic theory (see, e.g., \cite{Gilbarg2001}*{Section~8.12}) asserts that $\mathscr{A}$ 
        extends to a densely defined self-adjoint operator on $L^2(\mu)$ with discrete spectrum and $C^\infty(\S^{n-1})$ eigenfunctions, and its largest eigenvalue is simple with a positive eigenfunction. We have normalized $\mathscr{A}$ so that $\mathscr{A}h=h$, and thus, by orthogonality of the eigenfunctions, $h$ must correspond to the largest eigenvalue, which is necessarily $1$.
    \end{proof}

    The following lemma (taken from \cite{Shenfeld2019}*{Lemma~2.9}) describes the relation between the hyperbolicity of the quadratic form $Q(f,g)=\langle f,\mathscr{A} g\rangle$ and the spectral properties of $\mathscr{A}$. It can be regarded as a spectral version of \cref{lem:hyperbolic_iff_signature}.

    \begin{lem}\label{lem:hyperbolic_iff_spec}
        Let $(X,\langle{}\cdot{},{}\cdot{}\rangle)$ be a Hilbert space and $\mathscr{A}$ be a densely defined self-adjoint operator on $X$ that has discrete spectrum. Then the following are equivalent.
        \begin{enumerate}[label=\upshape(\alph*)]
            \item \label{lem:hyperbolic_iff_spec:hyperbolic}
            $\langle f, \mathscr{A}g \rangle^2 \geq \langle f, \mathscr{A}f \rangle\langle g, \mathscr{A}g \rangle$ for all $f,g \in\mathrm{dom}\mathscr{A}$ such that $\langle g, \mathscr{A}g \rangle > 0$.
            \item \label{lem:hyperbolic_iff_spec:spec}
            The positive eigenspace of $\mathscr{A}$ has dimension at most one.
        \end{enumerate}
    \end{lem}
    \begin{proof}
        Denoting by $\mathrm{spec}(\mathscr{A}) \subseteq \R$ the spectrum of $\mathscr{A}$ and $P_\lambda$ the orthogonal projection onto the eigenspace to an eigenvalue $\lambda\in \mathrm{spec}(\mathscr{A})$,   
        the Hilbert space $X$ decomposes into a direct orthogonal sum of the eigenspaces of $\mathscr{A}$ and for all $f\in\mathrm{dom}(\mathscr{A})$,
        \begin{equation}\label{eq:hyperbolic_iff_spec:proof}
            \langle f, \mathscr{A}f \rangle
            = \sum_{\lambda \in \mathrm{spec}(\mathscr{A})} \lambda\, \langle P_{\lambda}f , P_{\lambda}f \rangle
        \end{equation}
        
        \ref*{lem:hyperbolic_iff_spec:hyperbolic}$\Rightarrow$\ref*{lem:hyperbolic_iff_spec:spec}:
        By \cref{eq:hyperbolic_iff_spec:proof}, the quadratic form $Q(f,g)=\langle f, \mathscr{A}g \rangle$ defines an inner product on the positive eigenspace of $\mathscr{A}$.
        Hence, whenever $f,g \in \mathrm{dom}\mathscr{A}$ lie in the positive eigenspace of $\mathscr{A}$, then $\langle f, \mathscr{A}g \rangle^2 \leq \langle f, \mathscr{A}f \rangle\langle g, \mathscr{A}g \rangle$
        by the Cauchy--Schwarz inequality.        
        By our assumption, equality is attained, and thus $f$ and $g$ are linearly dependent. 

        \ref*{lem:hyperbolic_iff_spec:spec}$\Rightarrow$\ref*{lem:hyperbolic_iff_spec:hyperbolic}:
        By our assumption, the positive eigenspace of $\mathscr{A}$ is spanned by one single element $g\in\mathrm{dom}\mathscr{A}$. Thus, we may define a linear form $\xi:\mathrm{dom}\mathscr{A} \to \R, \xi(f)=\langle f,g\rangle$. By \cref{eq:hyperbolic_iff_spec:proof}, the quadratic form $Q(f,g)=\langle f, \mathscr{A}g \rangle$ is negative semidefinite on $\ker\xi$. Thus, \cref{lem:hyperbolic_iff} yields that $Q$ satisfies the desired reverse Cauchy--Schwarz inequality.
    \end{proof}

    We are now ready to prove \cref{thmintro:Alexandrov-Fenchel}.    

    \begin{thm}\label{thm:A-F-ineq_zeta}
        Let $2 < j \leq n-1$ and $\zeta\in C^2(\S^{n-1})$ be so that for every $u\in\S^{n-1}$, the matrix $D^2\zeta(u)$ is positive semidefinite or spectrally $(n-j)$-balanced. Then for all $K,L\in\K^n$ and $\bm{M}=(M_1,\ldots,M_{j-2})\subseteq \K^n$,
        \begin{equation}\label{eq:A-F-ineq_zeta}
        \V(K,L,\bm{M},\zeta,{\BB}^{[n-j-1]})^2 
        \geq \V(K,K,\bm{M},\zeta,{\BB}^{[n-j-1]}) \V(L,L,\bm{M},\zeta,{\BB}^{[n-j-1]}).
        \end{equation}
    \end{thm}
    \begin{proof}
        By continuity of the mixed volume, it suffices to show that the quadratic form $\V({}\cdot{},{}\cdot{},\bm{M},\zeta,{\BB}^{[n-j-1]})$ is hyperbolic on $C^\infty(\S^{n-1})$ in the case where $D^2\zeta(u)$ is positive definite or strictly spectrally $(n-j)$-balanced for every $u\in\S^{n-1}$, and the convex bodies $K,L,M_1,\ldots,M_{j-2}$ are all of class $C^\infty_+$. Moreover, by translation invariance of the mixed volume, we may assume that $M_1$ contains the origin in its interior. At every point $u\in\S^{n-1}$, we have that 
        \begin{align*}
            &\frac{\D(D^2f,D^2h,D^2h_{M_2},\ldots,D^2h_{M_{j-2}},D^2\zeta,I_{u^{\perp}}^{[n-j-1]})^2 }{\D(D^2h,D^2h,D^2h_{M_2},\ldots,D^2h_{M_{j-2}},D^2\zeta,I_{u^{\perp}}^{[n-j-1]})} \\
            &\qquad \geq \D(D^2f,D^2f,D^2h_{M_2},\ldots,D^2h_{M_{j-2}},D^2\zeta,I_{u^{\perp}}^{[n-j-1]})
        \end{align*}
        by employing Alexandrov's inequality~\cref{eq:Alexandrov_matrix_ineq} or the mixed discriminant inequality~\cref{eq:Alexandrov_type_ineq}.
        
        Next, we define the differential operator $\mathscr{A}$ and measure $\mu$ as in \cref{eq:differential_operator} and proceed precisely as in \cite{Shenfeld2019}*{Lemma~3.1}. For every function $f\in C^\infty(\S^{n-1})$,         
        \begin{align*}
            \pair{\mathscr{A}f}{\mathscr{A}f}_{L^2(\mu)}
            &= \int_{\S^{n-1}} h\, \frac{\D(D^2f,D^2h,D^2h_{M_2},\ldots,D^2h_{M_{j-2}},D^2\zeta,I_{u^{\perp}}^{[n-j-1]})^2 }{\D(D^2h,D^2h,D^2h_{M_2},\ldots,D^2h_{M_{j-2}},D^2\zeta,I_{u^{\perp}}^{[n-j-1]})} \, d\omega \\
            &\geq \int_{\S^{n-1}} h\, \D(D^2f,D^2f,D^2h_{M_2},\ldots,D^2h_{M_{j-2}},D^2\zeta,I_{u^{\perp}}^{[n-j-1]}) \, d\omega \\
            &= \int_{\S^{n-1}} f \, \D(D^2f,D^2h,D^2h_{M_2},\ldots,D^2h_{M_{j-2}},D^2\zeta,I_{u^{\perp}}^{[n-j-1]}) \, d\omega \\
            &= \pair{f}{\mathscr{A}f}_{L^2(\mu)}.
        \end{align*}
        In particular, if $\lambda$ is an eigenvalue of $\mathscr{A}$ and $f$ is a corresponding eigenfunction, then $\lambda^2 \geq \lambda$, which is only possible if $\lambda \notin (0,1)$. That is, the spectrum of $\mathscr{A}$ has a gap between $0$ and $1$. According to \cref{prop:differential_operator}, the largest eigenvalue of $\mathscr{A}$ is $1$ which is also a simple eigenvalue, and thus, the positive eigenspace of $\mathscr{A}$ is one-dimensional.
        \cref{lem:hyperbolic_iff_spec} and identity \cref{eq:mixed_vol_diff_op} yield that the quadratic form $\V({}\cdot{},{}\cdot{},\bm{M},\zeta,{\BB}^{[n-j-1]})$  is hyperbolic on $C^\infty(\S^{n-1})$. This concludes the argument.
    \end{proof}

    \begin{rem}\label{rem:homogeneity_2}
        In the case where $j=2$ and no reference bodies occur, the mixed discriminant $\D(D^2f,D^2\zeta,I_{u^\perp}^{[n-3]})$ is the most explicit, so one might expect this case to be particularly easy. However, the mixed discriminant inequality that our approach requires turns out not to hold pointwise, hence this case appears to call for a completely different approach. In a forthcoming article, Brauner, Ortega-Moreno, and Schuster~\cite{Brauner2026} develop a new kind of argument to prove an inequality of the like of \cref{eq:A-F-ineq_zeta} which does not rely on pointwise curvature inequalities.
    \end{rem}

    \subsection{Consequences}
    
    The Alexandrov--Fenchel type inequality~\cref{eq:A-F-ineq_zeta} gives rise to a rich family of geometric inequalities. Just like in the classical Brunn--Minkowski theory, these can now be deduced by purely algebraic manipulations of log-concave sequences (see, e.g., \cite{Schneider2014}*{§7.4}). Below, we demonstrate this for three inequalities, a Minkowski type inequality and a Brunn--Minkowski type inequality, in additive and multiplicative form. In particular, we obtain~\cref{thmintro:Brunn-Minkowski}.

    \begin{thm}\label{thm:Minkowsi-ineq_zeta}
    Let $2 < j \leq n-1$ and $\zeta\in C^2(\S^{n-1})$ be as in \cref{thm:A-F-ineq_zeta}.
    Then for all $K,L\in\K^n$ and $0 < i < j$,
        \begin{equation}\label{eq:Minkowski-ineq_zeta}
            \V(K^{[i]},L^{[j-i]},\zeta,{\BB}^{[n-j-1]})^{j}
            \geq \V(K^{[j]},\zeta,{\BB}^{[n-j-1]})^{i}\, \V(L^{[j]},\zeta,{\BB}^{[n-j-1]})^{j-i}.
        \end{equation}
    \end{thm}
    \begin{proof}
        We denote $a_i=\V(K^{[i]},L^{[j-i]},\zeta,{\BB}^{[n-j-1]})$ for $i\in \{0,\ldots,j\}$. According to the Alexandrov--Fenchel type inequality~\cref{eq:A-F-ineq_zeta}, the finite sequence $(a_i)$ is log-concave, that is,
        $a_i^2 \geq a_{i-1}a_{i+1}$ for $0<i<j$. Consequently, 
        the sequence of quotients $q_i = a_i/a_{i+1}$ is decreasing, that is,
        $q_i \geq q_{i-1}$.
        Hence for $0<i<j$, we obtain
        \begin{equation*}
            \bigg(\frac{a_i}{a_j}\bigg)^{i}
            = (q_i \cdots q_{j-1})^{i}
            \geq q_{i-1}^{i(j-i)}
            \geq (q_0 \cdots q_{i-1})^{j-i}
            = \bigg(\frac{a_0}{a_i}\bigg)^{j-i}.
        \end{equation*}
        Rearranging this inequality yields  $a_i^{j} \geq a_j^{i} a_0^{j-i}$, as desired.
    \end{proof}

    \begin{thm}\label{thm:B-M-ineq_zeta:additive}
        Let $2 < j \leq n-1$ and $\zeta\in C^2(\S^{n-1})$ be as in \cref{thm:A-F-ineq_zeta}.
        Then for all $K,L\in\K^n$ and $\lambda \in (0,1)$,
            \begin{align}\label{eq:B-M-ineq_zeta:additive}
                \begin{split}
                &\V(((1-\lambda)K+\lambda L)^{[j]},\zeta,{\BB}^{[n-j-1]})^{\frac{1}{j}} \\
                &\qquad \geq (1-\lambda) \V(K^{[j]},\zeta,{\BB}^{[n-j-1]})^{\frac{1}{j}} + \lambda
                 \V(L^{[j]},\zeta,{\BB}^{[n-j-1]})^{\frac{1}{j}}.
                \end{split}
            \end{align}
    \end{thm}
    \begin{proof}
        We denote $a_i=\V(K^{[i]},L^{[j-i]},\zeta,{\BB}^{[n-j-1]})$ for $i\in \{0,\ldots,j\}$. By the polynomiality of mixed volumes and inequality~\cref{eq:Minkowski-ineq_zeta},
        \begin{align*}
            &\V(((1-\lambda)K+\lambda L)^{[j]},\zeta,{\BB}^{[n-j-1]})
            = \sum_{i=0}^{j} \binom{j}{i}(1-\lambda)^{i}\lambda^{j-i} a_i  \\
            &\qquad \geq \sum_{i=0}^{j} \binom{j}{i}(1-\lambda)^{i}\lambda^{j-i} a_j^{\frac{i}{j}}\, a_0^{\frac{j-i}{j}} 
            = \left( (1-\lambda)a_j^{\frac{1}{j}} + \lambda a_0^{\frac{1}{j}} \right)^{j}.
        \end{align*}
         Thus, we obtain the desired Brunn--Minkowski type inequality.
    \end{proof}

    \begin{thm}\label{thm:B-M-ineq_zeta}
    Let $2 < j \leq n-1$ and $\zeta\in C^2(\S^{n-1})$ be as in \cref{thm:A-F-ineq_zeta}.
    Then for all $K,L\in\K^n$ and $\lambda \in (0,1)$,
        \begin{align}\label{eq:B-M-ineq_zeta}
            \begin{split}
            &\V(((1-\lambda)K+\lambda L)^{[j]},\zeta,{\BB}^{[n-j-1]}) \\
            &\qquad \geq \V(K^{[j]},\zeta,{\BB}^{[n-j-1]})^{1-\lambda} \, \V(L^{[j]},\zeta,{\BB}^{[n-j-1]})^{\lambda}.
            \end{split}
        \end{align}
    \end{thm}
    \begin{proof}
        This follows immediately from \cref{eq:B-M-ineq_zeta:additive} by an application of the inequality between weighted arithmetic and geometric mean.
    \end{proof}

    \section{Mean section operators}
    \label{sec:mean_sec_ops}

    In this section, we prove \cref{thmintro:B-M_mean_sec}, establishing a Brunn--Minkowsi type inequality for the intrinsic volumes of mean section bodies.
    They were defined in \cref{eq:MSO_defi} up to normalization of $\mu_k$, which we now fix to be $\mu_k(\{E\in\mathrm{AGr}_k(\R^n) \mid E\cap\BB\neq\emptyset\})=\kappa_{n-k}$, where throughout $\kappa_n$ denotes the volume of the $n$-dimensional unit ball.
    As was pointed out in the introduction, one important tool is an integral representation involving a family of functions $g_k\in C^\infty(-1,1)$ that was discovered by Berg~\cite{Berg1969} in his work on the Christoffel problem~\cite{Christoffel1865}. For every convex body $K\in\K^n$ of dimension at least two,
    \begin{equation}
        h(K - s(K),u)
        = \int_{\S^{n-1}} g_n(\pair{u}{v}) \, dS_{1}(K,v),
        \qquad u\in\S^{n-1},
    \end{equation}
    where $s(K)=\kappa_n^{-1} \int h_K(u)u\, du$ is the \emph{Steiner point} of $K$ (cf.~\cite{Schneider2014}*{p.~50}).
    Goodey and Weil~\cites{Goodey1992,Goodey2014} showed that for $2\leq k\leq n-1$ and a convex body $K\in\K^n$ of dimension at least $n-k+2$,
    \begin{equation}\label{eq:MSO_integral_rep}
        h(\MSO_k K - s(\MSO_k K) ,u)
        =  c_{n,k} \int_{\S^{n-1}} g_k(\pair{u}{v}) \, dS_{n-k+1}(K,v),
        \qquad u\in\S^{n-1},
    \end{equation}
    where throughout $c_{n,k} = (k\kappa_k\kappa_{n-k})/((n-k+1)n\kappa_n)$.
    Notably, formula~\cref{eq:MSO_integral_rep} expresses $\MSO_k$ in terms of $g_k$, regardless of the dimension $n$ of the ambient space.

    For notational convenience, we introduce the convolution product on the unit sphere that is naturally induced by the action of the Lie group $\mathrm{SO}(n)$ (see, e.g., \cites{Grinberg1999,Takeuchi1994}). For our purposes however, the following modified definition suffices: Given a function $f\in C(\S^{n-1})$ and $g\in C(-1,1)$ such that $\int_{-1}^1 |g(t)|(1-t^2)^{\frac{n-3}{2}}\,dt < \infty$, the convolution $f\ast g \in C(\S^{n-1})$ is given by
    \begin{align}\label{eq:spherical_convol}
        (f\ast g)(u)
        = \int_{\S^{n-1}} g(\pair{u}{v}) f(v)\, dv,
        \qquad u\in\S^{n-1}.
    \end{align} 
    In order to see that that this is well-defined, switch to spherical cylinder coordinates (see, e.g.~\cite{Groemer1996}*{p.~9}) and note that the integral above can be written as
    \begin{align*}
        (f\ast g)(u)
        = \frac{1}{(n-1)\kappa_{n-1}} \int_{-1}^1 \int_{\S^{n-1}\cap u^\perp} f(tu+\sqrt{1-t^2}w)\, dw\, g(t)(1-t^2)^{\frac{n-3}{2}}\, dt,
    \end{align*}
    where throughout $\kappa_m$ denotes the volume of the unit ball in $\R^m$. 

    In light of integral representation~\cref{eq:MSO_integral_rep} for the mean section operators, we are particular interested in convolution products involving the functions $g_k \in C^\infty(-1,1)$ introduced by Berg~\cite{Berg1969} (which we define later). It was shown in \cite{Berg1969}*{Théorème~3.3} that
    \begin{equation}\label{eq:Berg_fct_limit}
        \lim_{t\to \pm 1} (1-t^2)^{\frac{k-2}{2}}g_k(t)
        = 0,
    \end{equation}
    which implies that for $2\leq k\leq n-1$, the expression $(1-t^2)^{\frac{n-3}{2}}g_k(t)$ is bounded, and thus, for $f\in C(\S^{n-1})$, we always have $f\ast g_k \in C(\S^{n-1})$. 

    The following easy yet valuable lemma, based on techniques of Lutwak~\cite{Lutwak1990}, connects mixed volumes of mean section bodies to functionals of area measures of the form~\cref{eq:functional_area_measure}.

    \begin{lem}\label{lem:Lutwaks_trick}
        Let $1\leq i\leq n$ and $2\leq k < n-1$. Whenever $K,M\in\K^n$ and $M$ is of class $C^2_+$, then
        \begin{equation}\label{eq:Lutwaks_trick}
            \V(\MSO_k K, M^{[i-1]}, {\BB}^{[n-i]})
            = \V(K^{[j]},\zeta,{\BB}^{[n-j-1]}),
        \end{equation}
        where $j=n-k+1$ and $\zeta = c_{n,k} S_{i-1}(M,{}\cdot{})\ast g_k$.
    \end{lem}
    \begin{proof}
        By continuity of the mixed volume and mean section operator, we may also assume $K\in\K^n$ to be of class $C^2_+$. Then the area measures $S_j(K,{}\cdot{})$ and $S_{i-1}(M,{}\cdot{})$ each have a continuous density.
        Combining \cref{eq:F_zeta_mixed_vol}, \cref{eq:MSO_integral_rep}, and the classical fact that $S_{i-1}(M,{}\cdot{})$ integrates linear functions to zero, we have
        \begin{align*}
            &V(\MSO_k K, M^{[i-1]}, {\BB}^{[n-i]})
            = \frac{1}{n} \int_{\S^{n-1}} h(\MSO_k K,u)\, dS_{i-1}(M,u) \\
            &\qquad = \frac{c_{n,k}}{n} \int_{\S^{n-1}} \int_{\S^{n-1}} g_k(\pair{u}{v})\, dS_j(K,v) \, dS_{i-1}(M,u).
        \end{align*}
        By Fubini's theorem, we may change the order of integration, which yields
        \begin{align*}
            &V(\MSO_k K, M^{[i-1]}, {\BB}^{[n-i]})
            = \frac{c_{n,k}}{n} \int_{\S^{n-1}} \int_{\S^{n-1}} g_k(\pair{u}{v}) \, dS_{i-1}(M,u) \, dS_j(K,v) \\
            &\qquad = \frac{1}{n} \int_{\S^{n-1}} \zeta(v) \, dS_j(K,v) 
            = V(K^{[j]},\zeta,{\BB}^{[n-j-1]}),
        \end{align*}
        where the final equality is again due to \cref{eq:F_zeta_mixed_vol}.
    \end{proof}

     \subsection{Berg's functions} 

    In light of the previous lemma, we would like to determine when a convolution product $\zeta =f\ast g_k$ meets the requirements of~\cref{thmintro:Brunn-Minkowski}. In order to examine the properties of the matrix $D^2\zeta(u)$, one would naively interchange integration and differentiation to exploit known properties of $g_k$. To this end, we turn to a discussion of the functions $g_k$ and their regularity.
    The functions $g_k\in C^{\infty}(-1,1)$, for $k\geq 2$, were defined by Berg~\cite{Berg1969} recursively (with a different normalization) as follows.
    \begin{align*}
        g_2(t) &= \frac{1}{2\pi}(\pi-\arccos t)(1-t^2)^{\frac{1}{2}} - \frac{1}{4\pi}t, \\
        g_3(t) &= \frac{1}{2\pi} \bigg(1 + t\log(1-t) + \bigg(\frac{4}{3}-\log2\bigg) t\bigg), \\
        g_{k+2}(t) &= \frac{k+1}{2\pi}g_{k}(t) + \frac{k+1}{2\pi(k-1)}tg_{k}'(t) + \frac{k+1}{2\pi k\kappa_k}t
    \end{align*}
    The key property of Berg's functions, which we will use repeatedly, is that they satisfy the following linear second order differential equation. For all $k\geq 2$ and $t\in (-1,1)$,
    \begin{equation}\label{eq:Berg_fct_ODE}
        \frac{1}{k-1}(1-t^2)g_k''(t) - tg_k'(t) + g_k(t) = - \frac{1}{\kappa_k}t,
    \end{equation}
    as can be shown by induction on $k\geq 2$. Regarding their behavior at the endpoints, $g_2\in C[-1,1]$ and for $k\geq 3$, we have $g_k \in C[-1,1)$ and $\lim_{t\to 1}g_k(t) = -\infty$. That is, all of these functions but $g_2$ have a singularity at $t=1$ the order of which depends on $k\geq 3$, compare~\cref{eq:Berg_fct_limit}. A refined description of Berg's functions and their higher order derivatives at $t=\pm 1$ was obtained in \cite{Brauner2024}*{Lemma~4.3}.

    \begin{lem}[{\cite{Brauner2024}}]
        For all $k\geq 2$ and $m\geq 0$ where $(k,m)\notin \{(2,0),(3,0)\}$,
        \begin{align}
            \lim_{t\to -1} (1-t^2)^{\frac{k-3}{2} + m} g_k^{(m)}(t) &= 0, 
            \label{eq:Berg_fct_limit_left} \\
            \lim_{t\to +1} (1-t^2)^{\frac{k-3}{2} + m} g_k^{(m)}(t) &= - \frac{(k-1)(2\pi)^{m-2} }{\kappa_{k+2m-5}}. \label{eq:Berg_fct_limit_right}
        \end{align}
    \end{lem}
    From the differential equation and behavior at the endpoints, one can deduce that Berg's functions are strictly concave. 
    Let us note that the first part of the following lemma was already shown in \cite{Goodey2014}*{Lemma~4.5}.
    
    \begin{lem}\label{lem:Berg_fct_concave}
        For all $k\geq 2$ and $t\in (-1,1)$,
        \begin{equation}\label{eq:Berg_fct_concave}
            g_k(t) - t g_k'(t) > 0
            \qquad\text{and}\qquad 
            g_k''(t) < 0.
        \end{equation}
        In particular, for each $k\geq 2$, the function $g_k$ is strictly concave.
    \end{lem}
    \begin{proof}
        First, we fix $k\geq 2$ and claim that for all $t\in (-1,1)$
        \begin{equation}\label{eq:Berg_fct_integral_identity}
            (1-t^2)^{\frac{k-1}{2}}(g_k(t) - t g_k'(t))
            = \frac{k-1}{\kappa_k} \int_{-1}^t s^2(1-s^2)^{\frac{k-3}{2}}\, ds.
        \end{equation}
        Taking a derivative of the left hand side and then employing the differential equation~\cref{eq:Berg_fct_ODE} gives
        \begin{align*}
            &\frac{d}{dt}\left[ (1-t^2)^{\frac{k-1}{2}}(g_k(t) - t g_k'(t)) \right] \\
            &\qquad = - t(1-t^2)^{\frac{k-3}{2}}\left( (1-t^2)g_k''(t) +(k-1)(g_k(t)- tg_k'(t)) \right)
             = \frac{k-1}{\kappa_k} t^2(1-t^2)^{\frac{k-3}{2}}.
        \end{align*}
        Hence, the left and right hand side of \cref{eq:Berg_fct_integral_identity} can only differ by a constant. By~\cref{eq:Berg_fct_limit_left},
        \begin{equation*}
            \lim_{t\to \pm 1} (1-t^2)^{\frac{k-1}{2}}g_k(t)
            = 0
            \qquad\text{and}\qquad
            \lim_{t\to \pm 1} (1-t^2)^{\frac{k-1}{2}}g_k'(t)
            = 0,
        \end{equation*}
        hence the left hand side of \cref{eq:Berg_fct_integral_identity} tends to zero as $t\to -1$. This proves the integral identity~\cref{eq:Berg_fct_integral_identity}, and as a direct consequence, $g_k(t) - tg_k'(t)>0$.
        Moreover, if $t\in (-1,0)$,
        \begin{align*}
            g_k(t) - tg_k'(t)
            &= \frac{k-1}{\kappa_k} (1-t^2)^{- \frac{k-1}{2}}\int_{-1}^t s^2(1-s^2)^{\frac{k-3}{2}}\, ds \\
            & \geq  \frac{k-1}{\kappa_k} (-t)(1-t^2)^{- \frac{k-1}{2}}\int_{-1}^t (-s)(1-s^2)^{\frac{k-3}{2}}\, ds
            = - \frac{1}{\kappa_k} t,
        \end{align*}
        and thus, 
        $g_k(t)- tg_k'(t) > \frac{1}{\kappa_k} \max\{0,-t\}$
        for all $t\in (-1,1)$. Combining this inequality again with the differential equation~\cref{eq:Berg_fct_ODE} shows that
        \begin{align*}
            - \frac{1}{k-1}(1-t^2)g_k''(t)
            = (g_k(t)- tg_k'(t)) + \frac{1}{\kappa_k}t 
            > \frac{1}{\kappa_k} (\max\{0,-t\} + t)
            \geq 0.
        \end{align*}
        Therefore, $g_k''(t)<0$ for all $t\in (-1,1)$ and the function $g_k$ is strictly concave.
    \end{proof}

    Given a function $g\in C^2(-1,1)$, we write
    $g^v(u)=g(\pair{u}{v})$. 
    Then the second order derivatives of the one-homogeneous extension of $g^v$ can be obtained by a straightforward computation and are given by the following lemma.

    \begin{lem}[{\cite{OrtegaMoreno2021}}]\label{{lem:Hessian_zonal}}
        Let $g\in C^2(-1,1)$ and $v\in\S^{n-1}$. Then $g^{v} \in C^2(\S^{n-1}\setminus\{\pm v\})$ and for all $u\in \S^{n-1}\setminus\{\pm v\}$,
        \begin{align}\label{eq:Hessian_zonal}
            D^2g^{v}(u)
            = g''(\pair{u}{v})(P_{u^\perp}v \otimes P_{u^\perp}v) + (g(\pair{u}{v}) - \pair{u}{v}g'(\pair{u}{v}))P_{u^\perp}.
        \end{align}
    \end{lem}

    Since $g_k$ (for $k\geq 3$) has a singularity at $t=1$, the function $g_k^v$ and all its higher order derivatives are singular in $u=v$. Hence, interchanging differentiation and integration is a delicate issue that requires the theory developed in \cite{Brauner2023}*{Section~3}. We summarize the relevant information in the following lemma.

    \begin{lem}\label{lem:Berg_fct_convol_trafo}
        Let $2\leq k\leq n-1$. Then $g_k^v \in W^{2,1}(\S^{n-1})$ for every $v\in\S^{n-1}$. Moreover, whenever $f\in C(\S^{n-1})$, then $ f \ast g_k \in C^2(\S^{n-1})$ and
        \begin{equation}\label{eq:Berg_fct_convol_trafo}
            D^2[f\ast g_k](u)
            = \int_{\S^{n-1}} f(v) D^2g_k^v(u) \, dv,\qquad u\in\S^{n-1}.
        \end{equation}
    \end{lem}
    \begin{proof}
        Since $g_k\in C^{\infty}(-1,1)$, combining \cite{Brauner2023}*{Proposition~3.7} and \cite{Brauner2023}*{Theorem~3.15} gives that is sufficient to prove
        \begin{equation*}
            \int_{-1}^1 (1-t^2)^{\frac{n-3}{2}}|g_k(t)| \, dt < \infty
            \qquad\text{and}\qquad
            \int_{-1}^1  (1-t^2)^{\frac{n-1}{2}}|g_k''(t)|\, dt < \infty.
        \end{equation*}
        The finiteness of the first integral has already been discussed and follows from \cref{eq:Berg_fct_limit}.
        For the finiteness of the second integral, note that by $\cref{eq:Berg_fct_limit_left}$ and $\cref{eq:Berg_fct_limit_right}$,
        \begin{align*}
            \lim_{t\to -1} (1-t^2)^{\frac{k+1}{2}}g_k''(t) = 0 \qquad\text{and}\qquad
            \lim_{t\to +1} (1-t^2)^{\frac{k+1}{2}}g_k''(t) = -\frac{k-1}{\kappa_{k-1}}.
        \end{align*}
        Since $k\leq n-1$, it follows that $(1-t^2)^{\frac{n-1}{2}}|g_k(t)|\leq A_k (1-t^2)^{-\frac{1}{2}} $ for some constant $A_k\geq 0$. To conclude the proof, note that the function $(1-t^2)^{-\frac{1}{2}}$ is integrable.
    \end{proof}

    \subsection{Proof of \texorpdfstring{\cref{thmintro:B-M_mean_sec}}{Theorem~D}}

    We can now collect our findings about Berg's functions to show that the convolution product $f\ast g_k$ meets the requirements of~\cref{thmintro:Brunn-Minkowski} whenever $f\in C(\S^{n-1})$ is nonnegative and orthogonal to all linear functions. This is the content of the following key proposition.
    
     \begin{prop}\label{prop:Berg_fct_convol_balanced}
        Let $2\leq k \leq n-1$ and $f\in C(\S^{n-1})$ be nonnegative and orthogonal to all linear functions. Then $\zeta = f \ast g_k$ is a $ C^2(\S^{n-1})$~function and for all $u\in\S^{n-1}$,
        \begin{equation}\label{eq:Berg_fct_convol_balanced}
            \tr D^2\zeta(u)
            \geq (n-k)\lambda_{\max}(D^2\zeta(u)).
        \end{equation}
        In particular, if $2\leq k < n-1$, then $D^2\zeta(u)$ is spectrally $(k-1)$-balanced.
    \end{prop}
    \begin{proof}
        First, we take a direction $v\in\S^{n-1}$ and examine the eigenvalues of the Hessian $D^2g_k^v(u)\in\Sym^2(\R^n)$ at a given point $u\in\S^{n-1}\setminus \{\pm v\}$. By one-homogeneous extension of $g_k^v$, it has a zero eigenvalue in radial direction. Denoting
        \begin{equation*}
            a_k(t) = g_k(t) - t g_k'(t)
            \qquad\text{and}\qquad
            b_k(t) = (1-t^2)g_k''(t) + g_k(t) - t g_k'(t),
        \end{equation*}
        formula~\cref{eq:Hessian_zonal} shows that on $u^\perp$, the matrix $D^2g_k^v(u)$ has the eigenvalue $a_k(\pair{u}{v})$ with multiplicity~$n-2$, and $b_k(\pair{u}{v})$ with multiplicity one. \cref{lem:Berg_fct_concave} shows that $a_k(t)>0$ and $a_k(t)>b_k(t)$ for all $t\in (-1,1)
        $. Hence, for $v\in\S^{n-1}$ and $u\in\S^{n-1}\setminus\{\pm v\}$, the maximal eigenvalue of $D^2 g_k^v(u)$ is given by
        $\lambda_{\max}(D^2g_k^{v}(u))=a_k(\pair{u}{v})$. Employing the differential equation~\cref{eq:Berg_fct_ODE}, we have
        \begin{align*}
             \tr D^2g_k^v(u)
             = (n-2)a_k(\pair{u}{v}) + b_k(\pair{u}{v})
             = (n-k)\lambda_{\max}(D^2g_k^v(u)) - \frac{k-1}{\kappa_k}\pair{u}{v}.
        \end{align*}
        
        Next, take $f\in C(\S^{n-1})$ to be nonnegative and orthogonal to all linear functions. According to~\cref{lem:Berg_fct_convol_trafo}, the convolution product $\zeta=f\ast g_k$ is a $C^2(\S^{n-1})$ function and the Hessian of its one-homogeneous extension is given by~\cref{eq:Berg_fct_convol_trafo}. Then for every point $u\in\S^{n-1}$ and unit vector $w\in u^\perp$, we have
        \begin{align*}
            &\tr D^2\zeta(u)
            = \int_{\S^{n-1}} f(v) \tr D^2g_k^v(u)\, dv 
            = (n-k) \int_{\S^{n-1}} f(v) \lambda_{\max}(D^2g_k^v(u)) \, dv \\
            &\qquad \geq (n-k) \int_{\S^{n-1}} f(v) \pair{w}{ D^2g_k^v(u) w} \, dv 
            = (n-k) \pair{w}{ D^2\zeta(u) w }.
        \end{align*}
        Choosing $w$ to be an eigenvector corresponding to the largest eigenvalue of $D^2\zeta(u)$, we obtain~\cref{eq:Berg_fct_convol_balanced}.
        To show that $D^2\zeta(u)\in\Sym^2(u^\perp)$ is a spectrally $(k-1)$-balanced matrix, take a subset
        $I\subseteq[n-1]$ of size $|I|=k$ and note that
        \begin{align*}
            \sum_{i\in I} \lambda_i(D^2\zeta(u))
            &= \tr(D^2\zeta(u)) - \sum_{i\in [n-1]\setminus I} \lambda_i(D^2\zeta(u)) \\
            &\geq (n-k)\lambda_{\max}(D^2\zeta(u)) - (n-k-1)\lambda_{\max}(D^2\zeta(u)) \\
            &= \lambda_{\max}(D^2\zeta(u)).
        \end{align*}
        This concludes the argument.        
    \end{proof}

    Now we have almost all of the ingredients prepared to prove~\cref{thmintro:B-M_mean_sec}. We also require a generalized Minkowski and generalized Brunn--Minkowski inequality.
    For all $1\leq i\leq n-1$ and $C,M\in\K^n$,
    \begin{equation}\label{eq:aux_Minkowski_type_ineq}
        \V(C,M^{[i-1]},{\BB}^{[n-i]})^{i}
        \geq \V(C^{[i]},{\BB}^{[n-i]}) \V(M^{[i]},{\BB}^{[n-i]})^{i-1}.
    \end{equation}
    If $C$ and $M$ are full-dimensional, equality holds if and only if they are homothetic.
    Moreover, for all $2\leq j\leq n-1$, bodies $K,L\in\K^n$, and $\lambda\in (0,1)$,
    \begin{equation}\label{eq:aux_Brunn-Minkowski_type_ineq}
        V_j((1-\lambda)K+\lambda L)
        \geq V_j(K)^{1-\lambda} V_j(L)^{\lambda}.
    \end{equation}
    If $K$ and $L$ are full-dimensional, equality holds if and only if they are translates.
    These geometric inequalities, along with their equality conditions, are classical (see, e.g., \cite{Schneider2014}*{Section~7.6}) and follow from the classical Alexandrov--Fenchel inequality in the same way as~\cref{eq:Minkowski-ineq_zeta} and~\cref{eq:B-M-ineq_zeta} follow from~\cref{eq:A-F-ineq_zeta}.
    
    \begin{thm}\label{thm:B-M_mean_sec}
        Let $1\leq i\leq n$ and $2\leq k < n-1$. Then for all $K,L\in\K^n$ and $\lambda\in (0,1)$,
        \begin{equation}\label{eq:B-M_mean_sec}
            V_i(\MSO_k((1-\lambda)K+\lambda L))
            \geq V_i(\MSO_k K)^{1-\lambda} V_i(\MSO_k L)^{\lambda}.
        \end{equation}
        If $K$ and $L$ are full-dimensional, equality holds if and only if they are translates.
    \end{thm}
    \begin{proof}
        First, let $j=n-k+1$, take a convex body $M\in\K^n$ of class $C^2_+$ and let $\zeta = c_{n,k} S_{i-1}(M,{}\cdot{})\ast g_k$. Then the area measure $S_{i-1}(M,{}\cdot{})$ has a continuous density and \cref{prop:Berg_fct_convol_balanced} implies that $\zeta$ is a $C^2(\S^{n-1})$ function and the matrix $D^2\zeta(u)$ is spectrally $(n-j)$-balanced at every point $u\in\S^{n-1}$. Hence, the Brunn--Minkowski type inequality from~\cref{thmintro:Brunn-Minkowski} yields
        \begin{align*}
            \V(((1-\lambda)K+\lambda L)^{[j]},\zeta,{\BB}^{[n-j-1]})
            \geq  \V(K^{[j]},\zeta,{\BB}^{[n-j-1]})^{1-\lambda} \, V(L^{[j]},\zeta,{\BB}^{[n-j-1]})^{\lambda}.
        \end{align*}
        By the mixed volume identity~\cref{eq:Lutwaks_trick}, we can rewrite this as
        \begin{align*}
            &\V(\MSO_k ((1-\lambda)K+\lambda L), M^{[i-1]}, {\BB}^{[n-i]}) \\
            &\qquad \geq \V(\MSO_k K, M^{[i-1]}, {\BB}^{[n-i]})^{1-\lambda} \, \V(\MSO_k L, M^{[i-1]}, {\BB}^{[n-i]})^{\lambda}.
        \end{align*}
        By continuity of the mixed volume, this holds for any $K,L,M\in\K^n$ and $\lambda\in (0,1)$ (without the assumption of $C^2_+$ regularity).
        By inequality \cref{eq:aux_Minkowski_type_ineq},
        \begin{align}
            \V(\MSO_k K, M^{[i-1]}, {\BB}^{[n-i]})^{i}
            &\geq \V((\MSO_k K)^{[i]}, {\BB}^{[n-i]}) \V(M^{[i]}, {\BB}^{[n-i]})^{i-1}, 
            \label{eq:B-M_mean_sec:proof1}\\
            \V(\MSO_k L, M^{[i-1]}, {\BB}^{[n-i]})^{i}
            &\geq \V((\MSO_k L)^{[i]}, {\BB}^{[n-i]}) \V(M^{[i]}, {\BB}^{[n-i]})^{i-1}.
            \label{eq:B-M_mean_sec:proof2}
        \end{align}
        Combining this with the inequality above, we obtain
        \begin{align*}
            &\V(\MSO_k ((1-\lambda)K+\lambda L), M^{[i-1]}, {\BB}^{[n-i]})^{i} \\
            &\qquad \geq \V((\MSO_k K)^{[i]}, {\BB}^{[n-i]})^{1-\lambda} \, \V((\MSO_k L)^{[i]}, {\BB}^{[n-i]})^{\lambda} \, \V(M^{[i]}, {\BB}^{[n-i]})^{i-1}.
        \end{align*}
        Observe that if $K$ and $L$ are full-dimensional, then $M=\MSO_k((1-\lambda)K+\lambda L)$ is full-dimensional as well. Now we divide both sides of the inequality by $\V(M^{[i]}, {\BB}^{[n-i]})^{i-1}$ and note that $\binom{n}{i}V({{}\cdot{}}^{[i]},{\BB}^{[n-i]})=\kappa_{n-i}V_i$ (see \cite{Schneider2014}*{Equation~(5.31)}). This yields the desired inequality \cref{eq:B-M_mean_sec} if $K$ and $L$ are full-dimensional; the general inequality follows from this by approximation.

        If $K$ and $L$ are translates, we have equality in~\cref{eq:B-M_mean_sec}. Conversely, suppose that $K$ and $L$ are full-dimensional and equality is attained in~\cref{eq:B-M_mean_sec}. Then equality must also hold in \cref{eq:B-M_mean_sec:proof1} and \cref{eq:B-M_mean_sec:proof2}, hence the equality conditions of \cref{eq:aux_Minkowski_type_ineq} entail that 
        \begin{equation*}
            \MSO_k K = \alpha M + x
            \qquad{and}\qquad
            \MSO_k L = \beta M + y
        \end{equation*}
        for some $\alpha,\beta>0$ and $x,y\in\R^n$. Moreover, from the equality in \cref{eq:B-M_mean_sec}, we see that $\alpha^{1-\lambda}\beta^{\lambda}=1$. Next, note that the first intrinsic volume of a convex body $C\in\K^n$ is given by $V_1(C)=\frac{1}{\kappa_{n-1}}\int_{\S^{n-1}}h_C(u)\, du$ (see \cite{Schneider2014}*{Equations~(5.31), (5.57)}). From the definition~\cref{eq:MSO_defi} of the mean section operators and Crofton's intersection formula~\cite{Schneider2014}*{Equation~(4.59)}, it follows that        
        \begin{equation*}
            V_1(\MSO_k C) = d_{n,j} V_j(C)
        \end{equation*}
        for every convex body $C\in\K^n$, where $d_{n,j}>0$ is a constant that can be computed explicitly from \cite{Schneider2014}*{Equation~(4.65)}. Applying this to $K$, $L$, and $(1-\lambda)K+\lambda L$ yields
        \begin{equation*}
            V_j(K) = \frac{\alpha}{d_{n,j}} V_1(M), \quad
            V_j(L) = \frac{\beta}{d_{n,j}}  V_1(M), \quad\text{and}\quad
            V_j((1-\lambda )K+\lambda L) = \frac{1}{d_{n,j}} V_1(M).
        \end{equation*}
        Thus, equality holds in the Brunn--Minkowski type inequality \cref{eq:aux_Brunn-Minkowski_type_ineq} and its equality conditions imply that $K$ and $L$ are translates of each other.
    \end{proof}

    For completeness, we also include the Brunn--Minkowski type inequality for mean section operators in its additive form.

    \begin{thm}\label{thm:B-M_mean_sec:additive}
        Let $1\leq i\leq n$, let $2\leq k < n-1$, and denote $j=n-k+1$. Then for all $K,L\in\K^n$ and $\lambda\in (0,1)$,
        \begin{equation}\label{eq:B-M_mean_sec:additive}
            V_i(\MSO_k((1-\lambda)K+\lambda L))^{\frac{1}{ij}}
            \geq (1-\lambda) V_i(\MSO_k K)^{\frac{1}{ij}} + \lambda V_i(\MSO_k L)^{\frac{1}{ij}}.
        \end{equation}
        If $K$ and $L$ are full-dimensional, equality holds if and only if they are homothetic.
    \end{thm}
    \begin{proof}
        We may assume that $K$ and $L$ are full-dimensional; the general inequality follows from this by approximation.
        Note that the functional $\mathcal{F}:\K^n\to\R$,
        $\mathcal{F}(K)=V_i(\MSO_k K)^{\frac{1}{ij}}$ is one-homogeneous and log-concave by \cref{eq:B-M_mean_sec}. We set
        \begin{align*}
             \tilde{K} = \frac{1}{\mathcal{F}(K)}K,
             \qquad \tilde{L}= \frac{1}{\mathcal{F}(L)} L, 
             \qquad\text{and}\qquad \tilde{\lambda} = \frac{\lambda \mathcal{F}(L)}{(1-\lambda) \mathcal{F}(K) + \lambda \mathcal{F}(L) }.
        \end{align*}
        Then the log-concavity of $\mathcal{F}$ entails that
        \begin{align*}
            \frac{\mathcal{F}((1-\lambda)K+\lambda L)}{(1-\lambda)\mathcal{F}(K) + \lambda\mathcal{F}(L)}
            = \mathcal{F}((1-\tilde{\lambda})\tilde{K} + \tilde{\lambda}\tilde{L})
            \geq \mathcal{F}(\tilde{K})^{1-\tilde{\lambda}}\mathcal{F}(\tilde{L})^{1-\tilde{\lambda}}
            = 1,
        \end{align*}
        yielding inequality~\cref{eq:B-M_mean_sec:additive}.

        If $K$ and $L$ are homothetic, we have equality in~\cref{eq:B-M_mean_sec:additive}. Conversely, if equality holds in~\cref{eq:B-M_mean_sec:additive}, then equality is attained in \cref{eq:B-M_mean_sec} by $\tilde{K}$, $\tilde{L}$, and $\tilde{\lambda}$. The associated equality condition states that $\tilde{K}$ and $\tilde{L}$ are translates, and thus, $K$ and $L$ are homothetic.
    \end{proof}

    We conclude this section with two remarks: Firstly, we point out that the family of inequalities~\cref{eq:B-M_mean_sec} does not follow trivially from a set inclusion that the definition of the mean section operators might suggest. Secondly, we note that our argument can be extended to also involve Minkowski endomorphisms. 

    \begin{rem}\label{rem:B-M_mean_sec:nontrivial}
        Given two convex bodies $K,L\in \K^n$ and $\lambda \in (0,1)$, it is easy to see that for any affine subspace $E\in\mathrm{AGr}_k(\R^n)$,
        \begin{equation*}
            ((1-\lambda)K + \lambda L) \cap E
            \supseteq (1-\lambda)(K\cap E) + \lambda (L\cap E).
        \end{equation*}
        From the definition~\cref{eq:MSO_defi} of the mean section operator $\MSO_k$ (where we tacitly integrate only over those subspaces intersecting the body in question),
        a naive computation might suggest that $\MSO_k((1-\lambda)K + \lambda L)$ contains the body $(1-\lambda)\MSO_k K + \lambda\MSO_k L$, which would imply inequality \cref{eq:B-M_mean_sec} directly from the monotonicity and log-concavity of intrinsic volumes.
        
        To illustrate that this is in fact very rarely the case, take
        $K=r_1 \BB$ and $L=r_2 \BB$ to be two balls around the origin with different radii $r_1 \neq r_2$. If $2\leq k\leq n-1$, then $\MSO_k$ is homogeneous of degree $j=n-k+1\geq 2$ and the function $[0,\infty)\ni r\mapsto r^{j}$ is strictly convex. Then we have
        \begin{align*}
            &\MSO_k((1-\lambda)K + \lambda L)
            = ((1-\lambda)r_1 + \lambda r_2)^{j} \MSO_k \BB \\
            &\qquad \subsetneq ((1-\lambda)r_1^{j} + \lambda r_2^{j}) \MSO_k \BB 
            = (1-\lambda) \MSO_k K + \lambda \MSO_k L.
        \end{align*}
    \end{rem}
    
    \begin{rem}
        We would like to point out that our argument can be used to obtain a slightly stronger inequality than~\cref{eq:B-M_mean_sec} involving an additional geometric operator. A \emph{Minkowski endomorphism} is a map $\Phi: \K^n\to\K^n$ that is Minkowski additive, continuous in the Hausdorff distance, and commutes with all rigid motions. Moreover, a Minkowski endomorphism $\Phi$ is called \emph{monotone} if $K\subseteq L$ implies $\Phi K\subseteq \Phi L$, and \emph{weakly monotone} if $K\subseteq L$ together with $s(K)=s(K)=0$ implies $\Phi K \subseteq \Phi L$.
        
        Kiderlen~\cite{Kiderlen2006} and Dorrek~\cite{Dorrek2017} showed that they can be described in terms of the spherical convolution: For every Minkowski endomorphism, there exists a unique signed measure $\mu_{\Phi}$ on $[-1,1]$, called the \emph{generating measure} of $\Phi$, such that
        $h_{\Phi K}=h_K\ast \mu_{\Phi}$ for all $K\in\K^n$. For this, one needs to extend definition~\cref{eq:spherical_convol} to a function $f\in C(\S^{n-1})$ and a signed measure $\mu$ on $[-1,1]$.
        
        A Minkowski endomorphism $\Phi$ is monotone if and only if $\mu_{\Phi}$ is a positive measure, and $\Phi$ is weakly monotone if and only if $\mu_{\Phi}$ is \emph{weakly positive}, that is, the sum of a positive measure and a linear function.
        Two less obvious facts due to Kiderlen~\cite{Kiderlen2006} and Dorrek~\cite{Dorrek2017} are that every weakly positive measure generates a Minkowski endomorphism, but not every Minkowski endomorphism is weakly monotone.
        
        If $\Phi:\K^n\to\K^n$ is a weakly monotone Minkowski endomorphism, then one can obtain a similar inequality as~\cref{eq:B-M_mean_sec}, where $\MSO_k$ is replaced by the composition $\Phi\circ \MSO_k$. The essential modification of the proof consists in applying the Brunn--Minkowski type inequality~\cref{eq:B-M-ineq_zeta} to the iterated convolution product
        $\zeta = c_{n,k}(S_{i-1}(M,{}\cdot{}) \ast g_k) \ast \mu$. 
        Following the same argument, one can show that for all $K,L\in\K^n$ and $\lambda\in (0,1)$,
        \begin{equation*}
            V_i(\Phi ( \MSO_k((1-\lambda)K+\lambda L)))
            \geq V_i(\Phi ( \MSO_k K))^{1-\lambda} V_i(\Phi  (\MSO_k L))^{\lambda}.
        \end{equation*}
    \end{rem}

\section*{Acknowledgments}

Firstly, we are obliged to Franz Schuster for presenting us with the problem of Brunn--Minkowski type inequalities involving mean section operators, and for continuously encouraging us to work on it.
We thank Ramon van Handel for the illuminating discussions that nudged us in the right direction.
We would also like to thank Christos Saroglou for giving the impetus to add \cref{rem:B-M_mean_sec:nontrivial}.

The first-named author was supported by the Deutsche Forschungsgesellschaft (DFG), project number \hyperlink{https://gepris.dfg.de/project/520350299}{520350299}. The second-named author was partially supported by grant PID2025-169631NB-I00, funded by the Spanish Ministry of Science, Innovation and Universities and the State Research Agency (MICIU/AEI/10.13039/501100011033).

\bibliography{references}
    
\end{document}